\documentclass[a4paper,twoside]{amsart}
\usepackage{amssymb}
\usepackage{amsmath}
\usepackage{amsthm}
\usepackage{bm}
\usepackage[non-compressed-cites]{amsrefs}

\usepackage{mathrsfs}
\usepackage{bbm}
\usepackage{stmaryrd}

\usepackage[normalem]{ulem}

\usepackage{enumerate}

\usepackage{tikz}
\usepackage{float}
\usepackage{pgfmath}
\usepackage{color}

\usepackage{graphicx}
\usetikzlibrary{positioning,patterns,calc}
\usepackage{caption}
\usepackage{subcaption}

\usepackage{comment}

\newcommand{\real}{\mathbb{R}} 
\newcommand{\pinteger}{\mathscr{P}} 

\newcommand{\restrict}{\mathop{\llcorner}} 

\newcommand{\qspace}[2]{\mathbf{Q}_{#1}(#2)} 

\newcommand{\fmetric}{\mathcal{F}} 
\newcommand{\perm}[1]{\Pi_{#1}} 
\newcommand{\di}[1]{\llbracket #1 \rrbracket} 
\newcommand{\bflangle}{\boldsymbol{\langle}}
\newcommand{\bfrangle}{\boldsymbol{\rangle}}

\newcommand{\affspace}[2]{\mathbb{A}(#1,#2)} 
\newcommand{\homspace}[2]{\textup{Hom}(#1,#2)} 
\newcommand{\polyspace}[3]{\mathbb{P}^{#1}(#2,#3)} 

\newcommand{\homomorphism}[2]{\textup{Hom}(#1,#2)} 

\newcommand{\openball}[2]{\mathbf{U}\left(#1,#2\right)} 
\newcommand{\closedball}[2]{\mathbf{B}\left(#1,#2\right)} 

\newcommand{\up}[1]{\textup{#1}} 
\newcommand{\afa}[2]{{A}#1(#2)} 
\newcommand{\pfa}[2]{\up{P}^1_{#1}#2} 
\newcommand{\tfa}[1]{{\up{T}^1}{#1}} 
\newcommand{\dfa}[2]{{D}#1(#2)} 

\newcommand{\density}[3]{\Theta^{#1}(#2,#3)} 
\newcommand{\support}[1]{\textup{supp}\,#1} 

\newcommand{\dimension}[1]{\textup{dim }#1} 
\newcommand{\integrald}{\textup{d}} 
\newcommand{\cardinality}[1]{\textup{card }#1} 

\usepackage{etoolbox}
\makeatletter
\patchcmd{\@maketitle}
  {\ifx\@empty\@dedicatory}
  {\ifx\@empty\@date \else {\vskip3ex \centering\footnotesize\@date\par\vskip1ex}\fi
   \ifx\@empty\@dedicatory}
  {}{}
\patchcmd{\@adminfootnotes}
  {\ifx\@empty\@date\else \@footnotetext{\@setdate}\fi}
  {}{}{}
\makeatother

\newtheorem{theorem}{Theorem}[section]
\newtheorem{lemma}[theorem]{Lemma}
\newtheorem{remark}[theorem]{Remark}
\newtheorem{proposition}[theorem]{Proposition}
\newtheorem{corollary}[theorem]{Corollary}

\newtheorem{definition}[theorem]{Definition}

\newtheorem*{theorem*}{Theorem}

\newtheorem*{corollary*}{Corollary}

\newtheorem{claim}{\texttt{Claim}}
\makeatletter
\@addtoreset{claim}{theorem}
\makeatother

\begin{document}
\begin{abstract}
We prove that all reasonable definitions of multiple-valued functions with continuous derivatives are equivalent, which eliminates inconsistencies in the literature.
As an intermediate result we prove that affinely approximable functions defined on intervals admit a differentiable selection.
\end{abstract}
\title[On Almgren's Multiple-Valued Functions of Class $1$]{On Almgren's Multiple-Valued Functions of Class $1$}
\author{Nicolau S. Aiex}
\date{\today}
\address{88, Sec.4, Tingzhou Road, SE Building SE810, Taipei, 116059, Taiwan}
\email{nsarquis@math.ntnu.edu.tw}
\maketitle

\section{Introduction}
Multiple-valued functions were first introduced by F. J. Almgren Jr. in \cite{almgren2000} to prove its namesake Regularity Theorem \cite{almgren2000}*{Theorem 5.22} for area-minimizing currents of arbitrary codimension.
Since then it has been thoroughly used to study regularity theory of geometric objects.
We briefly cite only a few of many articles that successfully used this framework to prove regularity properties of minimizing currents, for example see \cites{delellis-nardulli-steinbruchel2024, hutchinson1986.2, hirsch-stuvard-valtorta2019, minter2024, lin2014, funk-hardt2020, skorobogatova2024, menne2010}.
This is arguably the most useful tool when it comes to regularity theory in high codimension.

In parallel with applications in geometric analysis, there has also been some development of the fundamentals of the theory of multiple-valued function.
We mention \cite{delellis-espadaro2011} where the authors have a more concise presentation of multiple-valued functions (also called $Q$-valued functions or multi-valued functions depending on author's preference).
Also \cite{goblet06} and \cite{delellis-grisanti-tilli2004} where the authors independently prove selection results for various regularity classes.
Finally we mention a few extensions of multiple-valued functions such as \cite{stuvard2022} where the notion of multiple-valued sections of vector bundles is defined and \cite{chou.h-s:arxiv2026} where the author develops a measure-theoretic generalization of multiple-valued functions on varifolds.

In this paper we would like to discuss multiple-valued functions of class $1$, that is, those which are differentiable in some sense and whose first derivative is continuous.
Since there is no addition structure in $Q$-space, there is not direct generalization of rate of change or a ratio definition of differentiation for multiple-valued functions.
However, Almgren used the notion of affinely approximable function (see Definition \ref{definition of affinely approximable 1}) which is a multiple-valued approximation by a affine maps whose linear component defines the multiple-valued derivative map (see Definition \ref{definition of derivative}).

The notion of multiple-valued functions of class $1$ (or $C^1$ in some articles) is somewhat inconsistent in the literature.
For example, in \cite{simon-wickramasekera2016}*{page 1216} the authors define a function of class $1$ as those whose multiple-valued linear derivative
\begin{equation*}
x\mapsto \dfa{f}{x}\in\qspace{q}{\homspace{V}{W}}
\end{equation*}
is continuous whereas in \cite{delellis-grisanti-tilli2004}*{Definition 3.6} and \cite{hutchinson1986.2}*{page 304} they define it as having continuous affine approximation in the space of multiple-valued affine maps
\begin{equation*}
x\mapsto A{f}(x)\in\qspace{q}{\affspace{V}{W}}.
\end{equation*}
Similarly, the definition of functions of class $(1,\alpha)$ in \cite{simon-wickramasekera2016}*{page 1216} requires the derivative to be $\alpha$-H\"older continuous and in \cite{delellis-grisanti-tilli2004}*{Definition 3.6} and \cite{hutchinson1986.2}*{page 304} they require the function to be continuous while the linear approximation to be $\alpha$-H\"older continuous simultaneously under the same permutation.

It is not entirely clear, a priori, that these two notions are equivalent because the definition of the distance function on $Q$-spaces is dependent on a choice of optimal permutation that minimizes the distance between single-valued points.
Since there is no order, the permutation that minimizes distance for the affine approximation and the one that minimizes distance for the corresponding linear maps may be completely different.
This issue is particularly relevant in the proof of \cite{hutchinson1986.2}*{Theorem 3.7} where the critical part of the argument is to find a permutation that simultaneously realizes the continuity of the function and the $\alpha$-H\"older continuity of the derivative altogether.
As a simple but concrete example we consider $\Phi,\Psi\in\qspace{2}{\affspace{\real}{\real}}$ multiple-valued affine maps over $\real$ where we identify the first coordinate with the constant component and the second coordinate with the linear component.
\begin{equation*}
\begin{aligned}
\Phi & = \di{(0,\frac{1}{2})}+\di{(1,3)}\\
\Psi & = \di{(0,1)}+\di{(1,\frac{1}{2})}
\end{aligned}
\end{equation*}
We note that $\fmetric(\Phi,\Psi)=3$ is realized by the identity permutation (see \ref{definition q points}) as is the distance of the constant components
\begin{equation*}
\begin{aligned}
\fmetric_{\qspace{2}{\real}}(\di{0}+\di{1},\di{0}+\di{1}) & =\min\{|0-1|+|1-0|, |0-0|+|1-1|\}\\
& = 0,
\end{aligned}
\end{equation*}
whereas the distance of the linear components
\begin{equation*}
\begin{aligned}
\fmetric_{\qspace{2}{\homspace{\real}{\real}}}(\di{\frac{1}{2}}+\di{3},\di{1}+\di{\frac{1}{2}}) & =\min\{|\frac{1}{2}-1|+|3-\frac{1}{2}|, |\frac{1}{2}-\frac{1}{2}|+|3-1|\}\\
& = 2,
\end{aligned}
\end{equation*}
is realized by the permutation that exchanges the order of points.
That is, $\fmetric(\Phi,\Psi)$ cannot be directly computed by simultaneously computing the corresponding distance of each individual component.

Our goal in this article is to prove that these definitions are indeed equivalent (see Theorem \ref{main theorem}) as well as another notion that is akin to Whitney's extension Theorem.
Despite being a fundamental result, it turns out to not be entirely trivial.

The main tools that we need are as follows.
A combinatorial Lemma (see Lemma \ref{combinatorial lemma}) that allow us to compare the distance of affine maps with its evaluation on a sufficiently dense sample of points.
If the sample of points is sufficiently large relatively to the number of possible permutations then we are able to overcome the issue discussed above.
A differentiable selection Theorem (see Theorem \ref{main theorem 1}) for affinely approximable functions on closed intervals which is an intermediate step between the continuous selection theorem of Almgren \cite{almgren2000}*{1.10(2)} and the class $1$ selection theorem in \cite{delellis-grisanti-tilli2004}*{Theorem 4.2}.
And an oscilation Lemma (see Lemma \ref{oscilation lemma}) that plays the role of a mean-value inequality or a Taylor polynomial error estimate.

The case of higher order derivatives has been mentioned in the literature (see for example \cites{delellis-grisanti-tilli2004,minter2025}) without a precise definition.
We believe that the correct definition is a bit more delicate.
One could attempt to define higher derivatives recursively, in which higher order affine approximations are defined on different spaces so that it is not entirely clear how to define the $k$-th derivative as a multiple-valued symmetric multilinear map.
Or we could define it by Taylor polynomial approximation similarly to \ref{main theorem}(2) below, which is a fairly strong definition but misses the comparison to the classical notion of derivatives.
We introduce the notion of multiple-valued polynomials in a general form hoping that this work will be further improved to cover the higher order cases.

This paper is organizes as follows.
In Section $2$ we briefly present fundamental facts about $Q$-space and multiple-valued functions.
In Section $3$ we introduce multiple-valued polynomials and prove the necessary combinatorial results.
In Section $4$ we prove the existence of local splittings and differentiable selections for functions defined on closed intervals.
In Section $5$ we prove the oscilation Lemma and our Main Theorem.

\textbf{Acknowledgments:} We would like to thank professor Ulrich Menne for many discussions on this topic.
The author was funded by NSTC grant 114-2811-M-003-030.

\section{Preliminaries}
\textbf{Notation.}
We will denote by $\pinteger$ the set of positive integers.
Given $q\in\pinteger$ we define $\perm{q}$ to be the set of all permutations of $\{1,\ldots,q\}$.
Whenever $(X,d)$ is a metric space, $x\in X$ and $r>0$ we denote 
\begin{equation*}
\openball{x}{r}=\{y\in X:d(x,y)<r\} \text{ and }\closedball{x}{r}.
\end{equation*}
If in addition we have arbitrary sets $A,B\subset X$, then we denote 
\begin{equation*}
\begin{aligned}
d(x,A)  = \inf\{d(x,a):a\in A\},&\, d(A,B)  = \inf\{d(a,b):a\in A, b\in B\},\\
\openball{A}{r}  = \{x\in X:d(x,A)<r\} &\text{ and }\closedball{A}{r}  = \{x\in X:d(x,A)\leq r\}.
\end{aligned}
\end{equation*}

Given a vector space $V$ and $a,b\in V$, we define $\overrightarrow{[a,b]}=\{a+t(b-a)\in V:t\in[0,1]\}$.
If in addition $V$ is normed, we will denote by $\|\cdot\|_V$ its norm and $d_V$ the induced metric but the subscript will be omitted whenever there is no ambiguity.
If $W$ is another vector space, $L:V\rightarrow W$ and $v\in V$, we write $\langle v,L\rangle\in W$ for the evaluation map of $L$ at $v$.
The canonical inner product of $\real^n$ will be denoted by $\bflangle\cdot,\cdot\bfrangle_{\real^n}$.

Let $n\in\pinteger$, $a,b\in\real$ with $a<b$ and $f:[a,b]\rightarrow\real^n$ a differentiable function, we denote
\begin{equation*}
D^+f(a)=\lim_{x\rightarrow a^+}\frac{f(x)-f(a)}{x-a} \text{ and } D^-f(b)=\lim_{x\rightarrow b^-}\frac{f(x)-f(b)}{x-b}
\end{equation*}
for the right-sided derivative at $a$ and left-sided derivative at $b$ respectively.

We begin by recalling the necessary definitions from Almgren's multiple-valued functions.
For more details that go far beyond the needs of this paper, we suggest \cite{almgren2000}*{Chapter $1$}.

\begin{definition}[\cite{almgren2000}*{Definition 1.1(1)-(2)}]\label{definition q points}
Let $W$ denote a finite dimensional normed vector space with norm $\|\cdot\|_W$ and $q\in\pinteger$.
Given $y\in W$ we denote $\di{y}=\boldsymbol{\delta}_y$  as in \cite{federer1969}*{2.5.19}, the 0-current associated to the Dirac measure on $W$ centered at $y$.
The space $\qspace{q}{W}$ is defined as the set of integral 0-currents $T$ on $W$ given by
\begin{equation*}
T=\sum_{\alpha=1}^q\di{y_\alpha}
\end{equation*}
for some $y_1,\ldots,y_q\in W$ not necessarily distinct.

If $q_1,q_2\in\pinteger$ and $S_i\in\qspace{q_i}{W}$ are given by $S_i=\sum_{\alpha=1}^{q_i}\di{y^i_\alpha}$ for each $i=1,2$, then we define $S=S_1+S_2\in\qspace{q_1+q_2}{W}$ as

\begin{equation*}
S=\sum_{i=1}^2\sum_{\alpha=1}^{q_i}\di{y^i_\alpha}.
\end{equation*}

For $T=\sum_{\alpha=1}^q\di{y_\alpha}$ and $S=\sum_{\alpha=1}^q\di{z_\alpha}$ in $\qspace{q}{W}$, we define the metric $\fmetric_{q,\|\|_{W}}$ as
\begin{equation*}
\fmetric_{q,\|\|_{W}}(T,S)=\inf\left\{\sum_{\alpha=1}^q\|y_\alpha-z_{\sigma(\alpha)}\|_W:\sigma\in\perm{q}\right\}.
\end{equation*}
The subscript will be omitted whenever it is unambiguous.

A function taking values in $\qspace{q}{W}$ is called a \textit{multiple-valued function} in $W$.
\end{definition}

\begin{remark}
An element $S\in\qspace{q}{W}$ is a sum of Dirac measures and we denote its $0$-density at a point as
\begin{equation*}
\density{0}{S}{y}=\lim_{r\rightarrow 0^+} S(\closedball{y}{r}).
\end{equation*}
So we can alternatively write $S$ as
\begin{equation*}
S=\sum_{y\in\support{S}}\density{0}{S}{y}\di{y}.
\end{equation*}
In particular, if $A\subset W$ satisfies $A\cap\support{S}\neq\emptyset$ and $\sum_{y\in A}\density{0}{S}{y}=q'$, then we can identify $S\restrict A\in\qspace{q'}{W}$.
\end{remark}

The following is an elementary lemma that can be seen as relating the $\fmetric$ metric with the Hausdorff distance of the support.

\begin{lemma}\label{lemma for local fmetric}
Let $W$ be a finite dimensional normed vector space , $q\in \pinteger$, $S\in\qspace{q}{W}$ and $\eta=\min\{\|y_1-y_2\|:y_1,y_2\in\support{S}, y_1\neq y_2\}$.
The following statements hold:
\begin{enumerate}[(1)]
\item If $0<\varepsilon<\frac{\eta}{2}$ and $T\in\qspace{q}{W}$ satisfies $S(\openball{y}{\varepsilon})=T(\openball{y}{\varepsilon})$ for all $y\in\support{S}$, then
\begin{equation*}
\fmetric(S,T)\leq q\varepsilon;
\end{equation*}
\item If $0<\varepsilon<\frac{\eta}{2}$ and $T\in\qspace{q}{W}$ satisfies $\fmetric(S,T)\leq\varepsilon$, then 
\begin{equation*}
S(\openball{y}{\varepsilon})=T(\openball{y}{\varepsilon})\text{ for all }y\in\support{S};
\end{equation*}
\item If $0<\varepsilon<\frac{\eta}{3}$ and $T_1,T_2\in\qspace{q}{W}$ satisfy $\fmetric(T_i,S)\leq\varepsilon$ for each $i=1,2$, then
\begin{equation*}
\fmetric_q(T_1,T_2)=\sum_{y\in\support{S}}\fmetric_{\density{0}{S}{y}}(T_1\restrict\openball{y}{\varepsilon},T_2\restrict\openball{y}{\varepsilon}).
\end{equation*}
\end{enumerate}
\end{lemma}
\begin{proof}
Let us first prove $(1)$.
Since $\varepsilon<\frac{\eta}{2}$, the collection $\{\openball{y}{\varepsilon}:y\in\support{S}\}$ is pairwise disjoint and $S(\openball{y}{\varepsilon})=\density{0}{y}{S}$ for every $y\in\support{S}$.

If $S(\openball{y}{\varepsilon})=T(\openball{y}{\varepsilon})$, then $\sum_{z\in\support{T}\cap\openball{y}{\varepsilon}}\density{0}{T}{z}=\density{0}{S}{y}$ and for each $y\in\support{S}$ we can find points $z^y_1,\ldots,z^y_{\density{0}{S}{y}}\in\openball{y}{\varepsilon}$ (possibly with repetition) such that
\begin{equation*}
T=\sum_{y\in\support{S}}\sum_{\beta=1}^{\density{0}{S}{y}}\di{z^y_\beta}.
\end{equation*}
It follows that
\begin{equation*}
\fmetric(S,T)\leq\sum_{y\in\support{S}}\sum_{\beta=1}^{\density{0}{S}{y}}\|z^y_\beta-y\|\leq q\varepsilon.
\end{equation*}

To prove $(2)$ we write $S=\sum_{\alpha=1}^q\di{y_\alpha}$ and $\mathcal{A}_y=\{\alpha\in\{1,\ldots,q\}:y_\alpha=y\}$ for each $y\in\support{S}$.
If $T=\sum_{\alpha=1}^q\di{z_\alpha}$, then there exists $\sigma\in\perm{q}$ such that
\begin{equation*}
\begin{aligned}
\varepsilon\geq\fmetric(S,T) & = \sum_{\alpha=1}^q\|y_{\alpha}-z_{\sigma(\alpha)}\|\\
              & = \sum_{y\in\support{S}}\sum_{\alpha\in\mathcal{A}_y}\|y-z_{\sigma(\alpha)}\|
\end{aligned}
\end{equation*}
Thus, for each $\alpha\in\mathcal{A}_y$ we have $z_{\sigma(\alpha)}\in\openball{y}{\varepsilon}$.
Reversely, if $z_{\sigma(\alpha)}\in\openball{y}{\varepsilon}$ then $\alpha\in\mathcal{A}_y$.
Indeed, for all $y'\in\support{S}$ with $y\neq y'$ we have
\begin{equation*}
\begin{aligned}
\|z_{\sigma(\alpha)}-y'\| & \geq \|y-y'\|-\|z_{\sigma(\alpha)}-y\|\\
                          & \geq \eta-\varepsilon\\
													& > \varepsilon.
\end{aligned}
\end{equation*}
If $\alpha\in\mathcal{A}_{y'}$ then it would contradict $\fmetric(S,T)\leq\varepsilon$ due to the choice of $\sigma$.

Therefore,
\begin{equation*}
\begin{aligned}
&T\restrict\openball{y}{\varepsilon} =\sum_{\alpha\in\mathcal{A}_y}\di{z_{\sigma(\alpha)}},\\
&T=\sum_{y\in\support{A}} T\restrict\openball{y}{\varepsilon}\textup{ and }\\
&T(\openball{y}{\varepsilon})=\cardinality{\mathcal{A}_y}=S(\openball{y}{\varepsilon}).
\end{aligned}
\end{equation*}

To prove $(3)$ we write $T_i=\sum_{\alpha=1}^q\di{z_\alpha^i}$ for each $i=1,2$ and for each $y\in\support{S}$ define $\mathcal{B}^i_y=\{\alpha\in\{1,\ldots,q\}:z^i_\alpha\in\openball{y}{\varepsilon}\}$.
From the previous proof we know that $T_1(\openball{y}{\varepsilon})=T_2(\openball{y}{\varepsilon})$ and
\begin{equation*}
T_i=\sum_{y\in\support{S}}T_i\restrict\openball{y}{\varepsilon}\,\text{ for }i=1,2.
\end{equation*}
In particular, if $y,y'\in\support{S}$ with $y\neq y'$, then $\mathcal{B}^i_y\cap\mathcal{B}^i_{y'}=\emptyset$,
\begin{equation*}
\bigcup_{y\in\support{S}}\mathcal{B}^i_{y}=\{1,\ldots,q\}\,\text{ and }\cardinality{\mathcal{B}^1_y}=\cardinality{\mathcal{B}^2_y}.
\end{equation*}
We may re-order either $\{z_\alpha^1\}$ or $\{z_\alpha^2\}$ so that $\mathcal{B}^1_y=\mathcal{B}^2_y$ for all $y\in\support{S}$ and write $\mathcal{B}_y$ for both sets.
Thus,
\begin{equation*}
T_i\restrict\openball{y}{\varepsilon}=\sum_{\alpha\in\mathcal{B}_y}\di{z^i_\alpha}
\end{equation*}
and
\begin{equation*}
\begin{aligned}
\fmetric_{\density{0}{S}{y}}(T_1\restrict\openball{y}{\varepsilon},T_2\restrict\openball{y}{\varepsilon}) =\min\{ & \sum_{\alpha\in\mathcal{B}_y}\|z^1_\alpha-z^2_{\sigma(\alpha)}\|:  \sigma\in\perm{q}\\
                                   & \quad\text{ and }\sigma(\mathcal{B}_y)=\mathcal{B}_y\}.
\end{aligned}
\end{equation*}

%

It follows directly that
\begin{equation*}
\fmetric_{q}(T_1,T_2)\leq\sum_{y\in\support{S}}\fmetric_{\density{0}{S}{y}}(T_1\restrict\openball{y}{\varepsilon},T_2\restrict\openball{y}{\varepsilon}).
\end{equation*}

To prove the opposite inequality we fix a permutation $\sigma\in\perm{q}$ such that
\begin{equation*}
\fmetric_{q}(T_1,T_2)=\sum_{\alpha=1}^q\|z_\alpha^1-z_{\sigma(\alpha)}^2\|.
\end{equation*}
\begin{claim}
$\sigma(\mathcal{B}_y)=\mathcal{B}_y$ for all $y\in\support{S}$.
\end{claim}
Suppose false, that is, there exists $y_0\in\support{S}$ such that $\sigma(\mathcal{B}_{y_0})\not\subset\mathcal{B}_{y_0}$ and $\sigma^{-1}(\mathcal{B}_{y_0})\not\subset\mathcal{B}_{y_0}$.
Hence, we may find $y',y''\in\support{S}\setminus\{y_0\}$ and $\alpha',\alpha''\in\mathcal{B}_{y_0}$ such that $\sigma(\alpha')\in\mathcal{B}_{y'}$ and $\sigma^{-1}(\alpha'')\in\mathcal{B}_{y''}$.
Observe that $y'$ and $y''$ could happen to be equal, as could $\alpha'$ and $\alpha'' $.

In other words, $z_{\alpha'}^1,z_{\alpha''}^2\in\openball{y_0}{\varepsilon}$, $z_{\sigma(\alpha')}^2\in\openball{y'}{\varepsilon}$ and $z_{\sigma^{-1}(\alpha'')}^1\in\openball{y''}{\varepsilon}$.
Note that $\fmetric_q(T_1,T_2)\leq\fmetric_q(T_1,S)+\fmetric_q(S,T_2)\leq 2\varepsilon$ and compute:
\begin{equation*}
\begin{aligned}
2\varepsilon & \geq \fmetric_q(T_1,T_2)\\
             & \geq \|z^1_{\alpha'}-z^2_{\sigma(\alpha')}\| + \|z^1_{\sigma^{-1}(\alpha'')}-z^2_{\alpha''}\|\\
						 & \geq \|y_0-y'\| - \|z^1_{\alpha'}-y_0\| - \|z^2_{\sigma(\alpha')}-y'\|\\
						 & \quad + \|y_0-y''\| - \|z^2_{\alpha''}-y_0\| - \|z^1_{\sigma^{-1}(\alpha'')}-y''\|\\
						 & \geq 2\eta-4\varepsilon,
\end{aligned}
\end{equation*}
which contradicts $\varepsilon<\frac{\eta}{3}$ and proves the claim.
Hence,
\begin{equation*}
\fmetric_{q}(T_1,T_2)\geq\sum_{y\in\support{S}}\fmetric_{\density{0}{S}{y}}(T_1\restrict\openball{y}{\varepsilon},T_2\restrict\openball{y}{\varepsilon}).
\end{equation*}
\end{proof}

\begin{definition}[\cite{almgren2000}*{Definition 1.1(9)-(10)}]\label{definition of affinely approximable 1}
Let $V,W$ be finite dimensional normed vector spaces with norms $\|\cdot\|_V,\|\cdot\|_W$ repectively.
A function $A:V\rightarrow W$ is called an \textit{affine map} if there exists $w\in W$ such that $\langle v,L\rangle=A(v)-w$ is a linear map and in particular $w=A(0)$.
We denote by $\affspace{V}{W}$ the vector space of all affine maps between $V$ and $W$ and endow it with the norm
\begin{equation*}
\|A\|_{\mathbb{A}(V,W)}=\|A(0)\|_W+\|L\|_{\homspace{V}{W}},
\end{equation*}
where $\|L\|_{\homspace{V}{W}}=\sup\{\|\langle v,L\rangle\|_W:v\in V,\|v\|_V\leq 1\}$.

A multiple-valued function $F:V\rightarrow\qspace{q}{W}$ is called an \textit{affine function} if there exist affine maps $A_1,\ldots,A_q\in\affspace{V}{W}$ such that for all $v\in V$ we have
\begin{equation*}
F(v)=\sum_{\alpha=1}^q\di{A_\alpha(v)}.
\end{equation*}
Let $\Omega\subset V$ be a subset and $a\in\Omega$.
We say that $f:\Omega\rightarrow\qspace{q}{W}$ is \textit{affinely approximable} at $a$ if $a$ is in the interior of $\Omega$ and there exists an affine function $F_a:V\rightarrow\qspace{q}{W}$ such that
\begin{equation*}
\lim_{x\rightarrow a}\|x-a\|_V^{-1}\fmetric(f(x),F_a(x))=0.
\end{equation*}
Since $F_a$ is uniquely defined, we denote $\afa{f}{a}=\sum_{\alpha=1}^q\di{A_\alpha}\in\qspace{q}{\affspace{V}{W}}$ the associated multiple-valued affine approximation.

We say that $f$ is \textit{strongly affinely approximable} at $a\in\Omega$ if $f$ is affinely approximable at $a$ and given affine maps $A_1,\ldots,A_q$ defining $Af(a)$ we have that $A_\alpha(a)=A_\beta(a)$ implies $A_\alpha=A_\beta$.

Let
\begin{equation*}
\Omega'=\{a\in\Omega:\text{f is affinely approximable at }a\},
\end{equation*}
so that the affine approximation map
\begin{equation*}
Af:\Omega'\rightarrow \qspace{q}{\affspace{V}{W}}
\end{equation*}
is well defined.
\end{definition}

\begin{definition}\label{definition of derivative}
Given $V$ and $W$ vector spaces, there is a natural surjective map $\iota:\affspace{V}{W}\rightarrow\homspace{V}{W}$ that takes $A\in\affspace{V}{W}$ into $\iota(A)\in\homspace{V}{W}$ defined by $\langle v,\iota(a)\rangle=A(v)-A(0)$.
We induce a map $\iota_\sharp:\qspace{q}{\affspace{V}{W}}\rightarrow\qspace{q}{\homspace{V}{W}}$ by pushforward of measures.

Whenever $q\in\pinteger$ and $\Omega\subset V$ is a set, $a\in\Omega$ and $f:\Omega\rightarrow\qspace{q}{W}$ is affinely approximable at $a$, we define the derivative of $f$ at $a$ as
\begin{equation*}
\dfa{f}{a}=\iota_\sharp Af(a).
\end{equation*}
Equivalently, if $\afa{f}{a}=\sum_{\alpha=1}^q\di{A_\alpha}$ and $L_\alpha\in\homomorphism{V}{W}$ is given by $\langle v,L_\alpha\rangle = A_\alpha(v)-A_\alpha(0)$, then $\dfa{f}{a}=\sum_{\alpha=1}^q\di{L_\alpha}$.
\end{definition}

\begin{definition}
Let $V,W$ be finite dimensional vector spaces, $\Omega\subset V$ be a set, $q\in\pinteger$ and $f:\Omega\rightarrow\qspace{q}{W}$ be a multiple-valued function.
We say that a function $g:\Omega\rightarrow W^q$ is a \textit{selection function} for $f$ in $\Omega$ if for all $x\in\Omega$ we have
\begin{equation*}
f(x)=\sum_{\alpha=1}^q\di{g_\alpha(x)},
\end{equation*}
where we identify $g(x)=(g_\alpha(x))_{\alpha=1,\ldots,q}$ and $g_\alpha:\Omega\rightarrow W$ for each $\alpha=1,\ldots,q$.
\end{definition}

The continuity or differentiability class of the selection is preserved onto the multiple-valued function that it represents.
This follows trivially from the definition of selection and affine approximation.
Although the contrary is not always true, for example, a continuous multiple-valued valued function may not admit a continuous selection (see for example \cite{goblet06}*{p.304}).

\begin{lemma}\label{differentiability is preserved}
Let $V,W$ be finite dimensional vector spaces, $\Omega\subset V$ be a set, $q\in\pinteger$, $f:\Omega\rightarrow\qspace{q}{W}$ be a multiple-valued function and $g:\Omega\rightarrow W^q$ be a selection for $f$ in $\Omega$.
If $g$ is differentiable on $\Omega$ then $f$ is affinely approximable on $\Omega$ and $G:\Omega\rightarrow\affspace{V}{W}^q$ defined as $G(x)=(g_\alpha(x)-\langle x,Dg_\alpha(x)\rangle+Dg_\alpha(x))_{\alpha=1,\ldots,q}$ is a selection for $Af$.
In particular, if $g$ is of class $1$ then $Af$ is continuous.
\end{lemma}

\section{Multiple-Valued Polynomial Functions}

Even though we only discuss first order approximations in this paper, we will introduce the notion of multiple-valued polynomials in a broader sense that may also be used for higher order approximations in the future.

\begin{definition}[\cite{federer1969}*{1.10.1-4}]
Let $V$ and $W$ be finite dimensional normed vector spaces, we define \textit{homogeneous polynomial functions} and \textit{polynomial functions} between $V$ and $W$ as in \cite{federer1969}*{1.10.4}.
If $k\in\pinteger\cup\{0\}$ then we denote
\begin{equation*}
\polyspace{k}{V}{W}=\bigoplus_{j=0}^k{\bigodot}^j(V,W)
\end{equation*}
as the space of all polynomial functions between $V$ and $W$ of degree at most $k$.
In particular, $\polyspace{1}{V}{W}=W\oplus\homspace{V}{W}$.

For $\phi_j\in{\bigodot}^j(V,W)$ we define the norm 
\begin{equation*}
\|\phi_j\|_{\odot^j}=\sup\{\|\phi_j(v_1,\ldots,v_j)\|_W:\|v_i\|_V\leq 1 \text{ for }i=1,\ldots,j\},
\end{equation*}
with the convention that $\|\phi_0\|_{\odot^0}=\|\phi_0\|_W$.
Whenever $\phi\in\polyspace{k}{V}{W}$ with coordinates $\phi_j\in\bigodot^j(V,W)$ we define $\|\phi\|_{\mathbb{P}^k}$ as the $L^1$-norm induced by $\{\|\cdot\|_{\odot^j}\}_{j=1,\ldots,k}$:
\begin{equation*}
\|\phi\|_{\mathbb{P}^k}=\sum_{j=0}^k\|\phi_j\|_{\odot^j}.
\end{equation*}

Furthermore, for $x\in V$ we write
\begin{equation*}
\phi(x)=\sum_{j=0}^k\frac{1}{j!}\langle x^j, \phi_j\rangle\in W,
\end{equation*}
where $x^j=x\odot\cdots\odot x\in\bigodot_j V$.
\end{definition}

\begin{remark}\label{remark on norm op of affine map}
It follows from \cite{federer1969}*{1.10.4,3.1.11} that for $\phi\in\polyspace{k}{V}{W}$ we have $\phi_j=D^j\phi(0)$ for each $j=0,\ldots,k$ and
\begin{equation*}
\|\phi\|_{\mathbb{P}^k}=\sum_{j=0}^k\|D^j\phi(0)\|_{\odot^j},
\end{equation*}
where $D^0\phi=\phi$.
In particular, $D^k\phi(0)=D^k\phi(a)=\phi_k$ for any $a\in V$.
\end{remark}

\begin{definition}
Let $V,W$ be finite dimensional normed vector spaces, $q\in\pinteger$, $k\in\pinteger\cup\{0\}$, 
$\Phi=\sum_{\alpha=1}^q\di{\phi_\alpha}\in\qspace{q}{\polyspace{k}{V}{W}}$ and $a\in V$.
We define the corresponding $\Phi$ induced multiple-valued polynomial centered at $a$, $\up{P}_{a,\Phi}:V\rightarrow\qspace{q}{W}$ as:
\begin{equation*}
\up{P}_{a,\Phi}(x)=\sum_{\alpha=1}^q\di{\phi_\alpha(x-a)}.
\end{equation*}
\end{definition}

Our goal is to define the multiple-valued equivalent of a Taylor polynomial of order $1$ as a first order approximation of a function.
We would like to write the corresonding polynomial with varying center because there is no canonical way of associating an affine map to the coefficients of the approximating polynomial.
As such, we will make a notational distinction between the affine approximation, the associated Taylor polynomial and its corresponding multiple-valued polynomial function.
At this point the alternative notation may seem redundant, but we recall that the spaces $\affspace{V}{W}$ and $\polyspace{1}{V}{W}$ are fundamentally different.
In the future this alternative notation will be more relevant when one considers higher order approximations.

\begin{definition}
Let $V,W$ be finite dimensional normed vector spaces, $q\in\pinteger$, $\Omega\subset V$, $a\in\Omega$ and $f:\Omega\rightarrow\qspace{q}{W}$ be a multiple-valued function that is affinely approximable at $a$ with $\afa{f}{a}=\sum_{\alpha=1}^q\di{A_\alpha}\in\qspace{q}{\affspace{V}{W}}$ and $L_\alpha\in\homomorphism{V}{W}$ the associated linear map to $A_\alpha$ for each $\alpha=1,\ldots,q$.

We define the pointwise $1$-jet of $f$ at $a$ as
\begin{equation*}
\tfa{f}(a)=\sum_{\alpha=1}^q\di{(A_\alpha(a),L_\alpha)}\in\qspace{q}{\polyspace{1}{V}{W}}.
\end{equation*}

Its induced multiple-valued polynomial centered at $a$ is the multiple-valued Taylor polynomial of $f$ at $a$ denoted by $\pfa{a}{f}:V\rightarrow\qspace{q}{W}$.
That is,
\begin{equation*}
\pfa{a}{f}=\up{P}_{a,\tfa{f}(a)}.
\end{equation*}
We also define the multiple-valued derivative map of $f$ at $a$, $\up{D}^1_af:V\rightarrow\qspace{q}{W}$ as
\begin{equation*}
\up{D}^1_af(v)=\sum_{\alpha=1}^q\di{\langle v,L_\alpha\rangle}.
\end{equation*}
\end{definition}

\begin{remark}
Note that if $f$ is affinely approximable at $a$ then
\begin{equation*}
\lim_{x\rightarrow a}\|x-a\|^{-1}_V\fmetric(f(x),\pfa{a}{f}(x))=0.
\end{equation*}
Reversely, if there exists $\Phi\in\qspace{q}{\polyspace{1}{V}{W}}$ such that
\begin{equation*}
\lim_{x\rightarrow a}\|x-a\|^{-1}_V\fmetric(f(x),\up{P}_{a,\Phi}(x))=0,
\end{equation*}
then $f$ is affinely approximable at $a$ with $\tfa{f}(a)=\Phi$.
\end{remark}
\begin{remark}
If $g:\Omega\rightarrow W^q$ and $M:\Omega\rightarrow \homspace{V}{W}^q$ are arbitrary selections for $f$ and $Df$ respectively, then for each $a$ we can find a permutation $\sigma_a\in\perm{q}$ such that
\begin{equation*}
\tfa{f}(a)=\sum_{\alpha=1}^q\di{(g_\alpha(a),M_{\sigma_a(\alpha)}(a))}.
\end{equation*}
\end{remark}

For real valued polynomials in one dimension, the space of polynomials as a vector space determined by its coefficients and as a subspace of continuous functions with pointwise convergence are equivalent.
The following couple of Lemmas will prove an analogue of this fact for finite families of vector-valued polynomial functions.

\begin{lemma}[Combinatorial Lemma]\label{combinatorial lemma}
Suppose $N\in\pinteger$ and $L\in[1,\infty)$.
There exists a constant $\Gamma_1(N,L)\in(0,\infty)$ with the following property.

If $V$ is a finite dimensional normed vector space, $a\in V$, $r\in(0,\infty)$, $X\subset\closedball{a}{r}$ and $\nu:X\rightarrow\{1,\ldots,N\}$ satisfy
\begin{equation*}
\sup\{d_V(y,X):y\in\closedball{a}{r}\}\leq \Gamma_1^{-1}r,
\end{equation*}
then there exist $b\in\closedball{a}{r}$, $s\in(0,\infty)$ and $i_0\in\{1,\ldots,N\}$ such that
\begin{enumerate}[(1)]
\item $\closedball{b}{s}\subset\closedball{a}{r}$;
\item $\Gamma_1 s\geq r$ and 
\item $\sup\{d_V(y,\{x\in X:\nu(x)=i_0\}):y\in\closedball{b}{s}\}\leq L^{-1}s$.
\end{enumerate}
\end{lemma}
\begin{proof}
When $N=1$ the result is trivial with $\Gamma_1=L$, $b=a$ and $s=r$.
Suppose $N\geq 2$, choose $\Gamma_1=(4L)^N$, define $c_1=a$ and $r_1=\frac{r}{2}$.
We can find a finite set $\mathcal{C}_1\subset\closedball{c_1}{r_1}$ such that $\closedball{c_1}{r_1}\subset\cup_{c\in\mathcal{C}_1}\closedball{c}{\frac{r_1}{2L}}$.
By choice of $\Gamma_1$, it follows that $X\cap\closedball{c}{\frac{r_1}{2L}}\neq\emptyset$ for all $c\in\mathcal{C}_1$.

Now, suppose that for some $k\in\{1,\ldots,N-1\}$ we have defined for each $i=1,\ldots,k$ a point $c_i\in\closedball{a}{r}$, $r_i=\frac{r}{2(4L)^{i-1}}$ and a finite set $\mathcal{C}_i\subset\closedball{a}{r}$ such that:	
\begin{enumerate}[(i)]
\item $c_{i}\in\closedball{c_{i-1}}{r_{i-1}}$ if $i\geq 2$;
\item $\mathcal{C}_i\subset \closedball{c_{i-1}}{r_{i-1}}$ if $i\geq 2$;
\item $\closedball{c_i}{r_i}\subset\cup_{c\in\mathcal{C}_i}\closedball{c}{\frac{r_i}{2L}}$ and
\item $\sup\{\nu(x):x\in X\cap\closedball{c}{\frac{r_i}{2L}},c\in \mathcal{C}_i\}\leq N-i+1$.
\end{enumerate}

From the choice of $\Gamma_1$ we have $X\cap\closedball{c_k}{r_k}\neq \emptyset$.
Furthermore, property $(iv)$ implies that either
\begin{enumerate}[(a)]
\item for all $c\in\mathcal{C}_k$ there exists $x_c\in X\cap\closedball{c}{\frac{r_k}{2L}}$ with $\nu(x_c)=N-k+1$ or
\item there exist $c_{k+1}\in\mathcal{C}_k$ such that
\begin{equation*}
\sup\{\nu(x):x\in X\cap\closedball{c_{k+1}}{\frac{r_k}{2L}}\}\leq N-k.
\end{equation*}
\end{enumerate}

If case $(a)$ holds then we can choose $b=c_k$, $s=r_k=\frac{r}{2(4L)^{k-1}}$ and $i_0=N-k+1$.
Properties $(1)$ and $(2)$ follow trivially.
Take $y\in\closedball{b}{s}$ then from property $(iii)$ there exists $c\in\mathcal{C}_k$ such that $y\in\closedball{c}{\frac{r_k}{2L}}$.
Hence $d_V(y,x_c)\leq \frac{r_k}{L}=L^{-1}s$, which proves $(3)$.
If case $(b)$ holds then define $r_{k+1}=\frac{r_k}{4L}=\frac{r}{2(4L)^k}$ and we can find a finite set $\mathcal{C}_{k+1}\subset\closedball{c_{k}}{r_k}$ so that $c_{k+1}$, $r_{k+1}$, $\mathcal{C}_{k+1}$ satisfy conditions $(i)$-$(iv)$.

Following the inductive construction above, either condition $(a)$ will hold for some $k\in\{1,\ldots,N-1\}$ or the construction will continue until $k+1=N$ in which case condition $(b)$ is impossible since by choice of $\Gamma_1$ we have $X\cap\closedball{c_N}{\frac{r_N}{2L}}\neq\emptyset$. That is, condition $(a)$ must hold and the choice $i_0=1$ concludes the proof.

\end{proof}

\begin{remark}
The above Lemma also holds true if we only require $V$ to be a locally compact metric space as this is the only property of finite dimensional vector spaces that was used in the proof.
\end{remark}

\begin{lemma}\label{polynomial lemma}
Suppose $m,N\in\pinteger$ and $k\in\pinteger\cup\{0\}$.
There exists a constant $\Gamma_2(m,k,N)\in(0,\infty)$ with the following property.
If $V,W$ are finite dimensional normed vector spaces, $\dimension{V}\leq m$, $a\in V$, $r\in(0,\infty)$, $\mathcal{U}\subset\polyspace{k}{V}{W}$ is a non-empty finite set with $\cardinality{\mathcal{U}}\leq N$ and $X\subset\closedball{a}{r}$ satisfy
\begin{equation*}
\sup\{d_V(y,X):y\in\closedball{a}{r}\}\leq\Gamma_2^{-1}r,
\end{equation*}
then
\begin{equation*}
\inf\{\sum_{j=0}^kr^j\|D^j\phi(a)\|_{\odot^j}:\phi\in\mathcal{U}\}\leq\Gamma_2\sup\{\inf\{\|\phi(x)\|_W:\phi\in\mathcal{U}\}:x\in X\}
\end{equation*}
\end{lemma}
\begin{proof}
Let $\Gamma_0(\max\{m,k\})>1$ be given by \cite{menne2019}*{Lemma 2.1} and $\Gamma_1(N,\Gamma_0)$ be given by Lemma \ref{combinatorial lemma}.
We will choose $\Gamma_2=\max\{(k+1)\Gamma_0(1+\Gamma_1)^k,\Gamma_1\}$.

Write $\mathcal{U}=\{\phi_1,\ldots,\phi_l\}$ for some positive integer $l\leq N$ and let $\nu:X\rightarrow\{1,\ldots,N\}$ be an arbitrary selection function satisfying
\begin{equation*}
\inf\{\|\phi(x)\|_W:\phi\in\mathcal{U}\}=\|\phi_{\nu(x)}(x)\|_W
\end{equation*}
for all $x\in X$.

It follows from Lemma \ref{combinatorial lemma} that there exist $b\in\closedball{a}{r}$, $s\in(0,\infty)$ and $i_0\in\{1,\ldots,N\}$ such that $\closedball{b}{s}\subset\closedball{a}{r}$, $\Gamma_1s\geq r$ and
\begin{equation*}
\sup\{d_V(y,\{x\in X:\nu(x)=i_0\}):y\in\closedball{b}{s}\}\leq\Gamma_0^{-1}s.
\end{equation*}

Put $X'=\{x\in X:\nu(x)=i_0\}\cap\closedball{b}{2s}$ and $\phi'=\phi_{i_0}$.
We note that \cite{menne2019}*{Lemma 2.1} holds true on $\closedball{b}{s}$ with respect to $X'\subset\closedball{b}{2s}$.
Therefore,

\begin{equation*}
\begin{aligned}
\sup\{s^i\|D^i\phi'(y)\|_{\odot^i}:i=0,\ldots,k,y\in\closedball{b}{s}\}  \leq\Gamma_0\sup\{\|\phi'(x)\|_W:x\in X'\}&\\
                                                            = \Gamma_0\sup\{\inf\{\|\phi(x)\|_W:\phi\in\mathcal{U}\}:x\in X'\}&\\
																													  \leq \Gamma_0\sup\{\inf\{\|\phi(x)\|_W:\phi\in\mathcal{U}\}:x\in X\}&.
\end{aligned}
\end{equation*}
Since $\Gamma_1 s\geq r$, we have
\begin{equation*}
\|D^i\phi'(b)\|_{\odot^i}\leq\frac{\Gamma_1^i}{r^i}\Gamma_0\sup\{\inf\{\|\phi(x)\|_W:\phi\in\mathcal{U}\}:x\in X\},
\end{equation*}
for all $i=0,\ldots, k$.
It follows from Taylor's Formula \cite{federer1969}*{1.10.4,3.1.10} that
\begin{equation*}
D^i\phi'(a)=\sum_{l=i}^k\frac{(a-b)^{l-i}}{(l-i)!}\lrcorner D^l\phi'(b).
\end{equation*}
Hence,
\begin{equation*}
\begin{aligned}
\|D^i\phi'(a)\|_{\odot^i} & \leq \sum_{l=i}^k\frac{r^{l-i}}{(l-i)!}\|D^l\phi'(b)\|_{\odot^l}\\
             & \leq \sum_{l=i}^k\frac{r^{l-i}}{(l-i)!}\frac{\Gamma_1^l}{r^l}\Gamma_0\sup\{\inf\{\|\phi(x)\|_W:\phi\in\mathcal{U}\}:x\in X\}\\
						 & \leq \frac{(1+\Gamma_1)^k\Gamma_0}{r^i}\sup\{\inf\{\|\phi(x)\|_W:\phi\in\mathcal{U}\}:x\in X\}
\end{aligned}
\end{equation*}
for all $i=0,\ldots,k$.
The result follows by observing that
\begin{equation*}
\inf\{\sum_{j=0}^kr^j\|D^j\phi(a)\|_{\odot^j}:\phi\in\mathcal{U}\}\leq(k+1)\max\{r^i\|D^i\phi'(a)\|_{\odot^i}:i=0,\ldots,k\}.
\end{equation*}
\end{proof}

Although we may view a point in $\qspace{q}{W}$ as an element of $W^q$, this inclusion depends on a choice of ordering.
Since a priori there is no preferential ordering for representatives of multiple-valued polynomial functions, we consider the inclusion on $W^q$ with respect to all possible permutations.
In other words, we may view multiple-valued polynomial functions as finite families of (single-valued) polynomial functions and use the combinatorial results above to show that even in the multiple-valued case, pointwise convergence of polynomials as functions is equivalent to the convergence of its coefficients.

\begin{lemma}\label{multivalued combinatorial lemma}
Suppose $m,q\in\pinteger$, $k\in\pinteger\cup\{0\}$.
There exists a positive constant $\Gamma_3(m,k,q)\in(0,\infty)$ with the following property.
If $r\in(0,\infty)$, $V,W$ are finite dimensional normed vector spaces, $a,b,c\in V$ and $\Phi,\Psi\in\qspace{q}{\polyspace{k}{V}{W}}$ are given by $\Phi=\sum_{\alpha=1}^q\di{\phi_\alpha}$, $\Psi=\sum_{\alpha=1}^q\di{\psi_\alpha}$, $\dimension{V}\leq m$ and $X\subset\closedball{c}{r}$ satisfy
\begin{equation*}
\sup\{d_V(y,X):y\in\closedball{c}{r}\}\leq\Gamma_3^{-1}r,
\end{equation*}
then
\begin{equation*}
\begin{aligned}
\inf\{\sum_{\alpha=1}^q\sum_{j=0}^k & r^j\|D^j{{\phi_{\alpha}}}(c-a)-D^j{{\psi_{\sigma(\alpha)}}}(c-b)\|_{\odot^j}:\sigma\in\Pi_q\}\leq\\
  &\quad\leq \Gamma_3\sup\{\fmetric_{q,\|\|_W}(\up{P}_{a,\Phi}(x),\up{P}_{b,\Psi}(x)):x\in X\}
\end{aligned}
\end{equation*}
and
\begin{equation*}
\begin{aligned}
\sup\{&\fmetric_{q,\|\|_W}(\up{P}_{a,\Phi}(x),\up{P}_{b,\Psi}(x)):x\in\closedball{c}{r}\}\leq\\
  &\quad\leq\inf\{\sum_{\alpha=1}^q\sum_{j=0}^k r^j\|D^j{{\phi_{\alpha}}}(c-a)-D^j{{\psi_{\sigma(\alpha)}}}(c-b)\|_{\odot^j}:\sigma\in\Pi_q\}.
\end{aligned}
\end{equation*}
\end{lemma}
\begin{proof}
Let $\Gamma_3=q\Gamma_2$ where $\Gamma_2=\Gamma_2(m,k,q!)$ is given by Lemma \ref{polynomial lemma}.
We begin by endowing $W^q$ with the $L^1$-norm on $W^q$ induced by $\|\cdot\|_W$, that is,
\begin{equation*}
\|(w_\alpha)_{\alpha=1,\ldots,q}\|_{W^q}=\sum_{\alpha=1}^q\|w_\alpha\|_W.
\end{equation*}
Given $(\zeta_\alpha)_{\alpha=1,\ldots,q}\in\odot^i(V,W)^q$ we may identify $\zeta\in\odot^i(V,W^q)$ defined as
\begin{equation*}
\zeta(x)=(\zeta_\alpha(x))_{\alpha=1,\ldots,q},
\end{equation*}
The corresponding induced norms $\|\cdot\|_{\odot^i(V,W^q)}$ and $\|\cdot\|_{\odot^i(V,W)}$ are related by
\begin{equation*}
\|\zeta\|_{\odot^i(V,W^q)}\leq\sum_{\alpha=1}^q\|\zeta_\alpha\|_{\odot^i(V,W)}\leq q\|\zeta\|_{\odot^i(V,W^q)}.
\end{equation*}
Now, for each $\sigma\in\Pi_q$ we use this identification to define the polynomial function $\xi_\sigma\in\polyspace{k}{V}{W^q}$ as 
\begin{equation*}
\xi_\sigma(x)=(\phi_\alpha(x-a)-\psi_{\sigma(\alpha)}(x-b))_{\alpha=1,\ldots,q}
\end{equation*}
and a subset of polynomial functions
\begin{equation*}
\mathcal{U}=\{\xi_\sigma:\sigma\in\Pi_q\}\subset\polyspace{k}{V}{W^q}.
\end{equation*}
It follows that
\begin{equation*}
\begin{aligned}
&\inf\{\|\xi(x)\|_{W^q}:\xi\in{\mathcal{U}}\}  =\fmetric(\up{P}_{a,\Phi}(x),\up{P}_{b,\Psi}(x))\text{ and }\\
&\sum_{\alpha=1}^q\|D^j{{\phi_{\alpha}}}(c-a)-D^j{{\psi_{\sigma(\alpha)}}}(c-b)\|_{\odot^j(V,W)}\leq q\|D^j\xi_\sigma(c)\|_{\odot^j(V,W^q)}
\end{aligned}
\end{equation*}
for each $j=0,\ldots,k$.
Hence, the first inequality follows directly by applying Lemma \ref{polynomial lemma} to $\mathcal{U}$.

To prove the second inequality take $\bar{\sigma}\in\Pi_q$ so that
\begin{equation*}
\begin{aligned}
 \inf\{\sum_{\alpha=1}^q\sum_{j=0}^k r^j\|D^j{{\phi_{\alpha}}}(c-a)- & D^j{{\psi_{\sigma(\alpha)}}}(c-b)\|_{\odot^j(V,W)}:\sigma\in\Pi_q\}  =\\
&\sum_{\alpha=1}^q\sum_{j=0}^k r^j\|D^j{{\phi_{\alpha}}}(c-a)-D^j{{\psi_{\bar\sigma(\alpha)}}}(c-b)\|_{\odot^j(V,W)}.
\end{aligned}
\end{equation*}
Then, for every $x\in\closedball{c}{r}$ we have
\begin{equation*}
\begin{aligned}
\inf\{\|\xi(x)\|_{W^q}:\xi\in\mathcal{U}\} & \leq \|{\xi_{\bar{\sigma}}}(x)\|_{W^q}\\
                           & = \left\|\sum_{j=0}^k\langle\frac{(x-c)^j}{j!},D^j{{\xi_{\bar{\sigma}}}}(c)\rangle\right\|_{W^q}\\
													 & \leq \sum_{j=0}^k\frac{r^j}{j!}\|D^j{{\xi_{\bar{\sigma}}}}(c)\|_{\odot^j(V,W^q)}\\
													 & \leq \sum_{j=0}^kr^j\|D^j{{\xi_{\bar{\sigma}}}}(c)\|_{\odot^j(V,W^q)}\\
													 & \leq \sum_{\alpha=1}^q\sum_{j=0}^k r^j\|D^j{{\phi_{\alpha}}}(c-a)-D^j{{\psi_{\bar\sigma(\alpha)}}}(c-b)\|_{\odot^j(V,W)}\\
													  = \inf&\{\sum_{\alpha=1}^q\sum_{j=0}^k r^j\|D^j{{\phi_{\alpha}}}(0)-D^j{{\psi_{\sigma(\alpha)}}}(a-b)\|_{\odot^j(V,W)}:\sigma\in\Pi_q\},
\end{aligned}
\end{equation*}
where in the second line we used the Taylor formula \cite{federer1969}*{1.10.4(p.46),3.1.11(p.221)}.
\end{proof}

\section{Differentiable Selections for multiple-valued functions on $\real$.}

In this section we will prove that affinely approximable functions always admit a differentiable selection when the domain is one-dimensional.

\begin{definition}
Let $n,q\in\pinteger$, $a,b\in\real$ with $a<b$ and $f:[a,b]\rightarrow\qspace{q}{\real^n}$ be a multiple-valued function.
We say that $f$ is \textit{affinely approximable at $a$} if there exists an affine function $F_a:\real\rightarrow\qspace{q}{\real^n}$ such that
\begin{equation*}
\lim_{x\rightarrow a^+}\|x-a\|^{-1}\fmetric(f(x),F_a(x))=0.
\end{equation*}
Analogously for affinely approximable at $b$.

If $A\subset \real$ is an arbitrary set, $I\subset A$ is an interval and $a\in I$, then we say that $f:A\rightarrow\qspace{q}{\real^n}$ is \textit{affinely approximable at $a$ relative to $I$} when either $a$ belongs to the interior of $I$ and $f$ is affinely approximable at $a$ as defined in Definition \ref{definition of affinely approximable 1} or when $a$ is an endpoint of $I$ and $f$ is affinely approximable at $a$ as defined above.
\end{definition}

The following three Lemmas are meant as building blocks to construct a differentiable selection on a closed interval once we have a local differentiable selection near every point.
Firstly we prove that two differentiable selections may be patched together.
Then we prove that local selections on a half open interval extend to a global selection.
Finally we show that the same holds on closed intervals.

\begin{lemma}\label{concatenation of selections}
Let $n,q\in\pinteger$, $a,b\in\real$ with $a<b$ and $f:[a,b]\rightarrow\qspace{q}{\real^n}$ be affinely approximable on $[a,b]$.
If $c\in(a,b)$, $g_1:[a,c]\rightarrow (\real^n)^q$ and $g_2:[c,b]\rightarrow (\real^n)^q$ are differentiable selections for $f|_{[a,c]}$ and $f|_{[c,b]}$ respectively, then there exists $\sigma\in\perm{q}$ such that $h:[a,b]\rightarrow(\real^n)^q$ defined as
\begin{equation*}
h_\alpha(x)=\left\{
\begin{aligned}
(g_1)_\alpha(x) &, \text{ if }x\in[a,c]\\
(g_2)_{\sigma(\alpha)}(x) &, \text{ if }x\in[c,b],
\end{aligned}
\right.
\end{equation*}
for each $\alpha=1,\ldots,q$ is a differentiable selection for $f$.
\end{lemma}
\begin{proof}
Since $g_1$ and $g_2$ are differentiable at $c$, then we have
\begin{equation*}
\sum_{\alpha=1}^q\di{((g_1)_\alpha(c),D^-(g_1)_\alpha(c))}=Af(c)=\sum_{\alpha=1}^q\di{((g_2)_\alpha(c),D^+(g_2)_\alpha(c))}.
\end{equation*}
In particular, there exists $\sigma\in\perm{q}$ such that $(g_1)_\alpha(c)=(g_2)_{\sigma(\alpha)(c)}$ and $D^-(g_1)_\alpha(c)=D^+(g_2)_{\sigma(\alpha)(c)}$ for all $\alpha=1,\ldots,q$, which proves that $h$ is differentiable at $c$.
\end{proof}

\begin{lemma}\label{local to global selection}
Let $n,q\in\pinteger$, $a,b\in\real$ with $a<b$ and $f:[a,b]\rightarrow\qspace{q}{\real^n}$ be affinely approximable on $[a,b)$.
Suppose that for all $x\in[a,b)$ there exists $\varepsilon>0$ and $g:[a,b)\cap\openball{x}{\varepsilon}\rightarrow (\real^n)^q$ a differentiable selection for $f|_{[a,b)\cap\openball{a}{\varepsilon}}$.
Then, there exists $h:[a,b):\rightarrow(\real^n)^q$ a differentiable selection for $f|_{[a,b)}$.
\end{lemma}
\begin{proof}
We begin by defining
\begin{equation*}
t_0=\sup\{t> a: \text{ there exists a differentiable selection for }f|_{[a,t)}\}.
\end{equation*}
It follows from the hypothesis that $t_0>a$.
If $t_0<b$, then we can find $\varepsilon>0$ with $t_0+\varepsilon<b$ and $t_0-\varepsilon>a$, $G_1:[a,t_0)\rightarrow(\real^n)^q$ a differentiable selection on $[a,t_0)$ and $G_2:(t_0-\varepsilon,t_0+\varepsilon)\rightarrow(\real^n)^q$ a differentiable selection on $(t_0-\varepsilon,t_0+\varepsilon)$.

Now we pick $c\in(a,t_0-\varepsilon)$ and $t_1\in(t_0,t_0+\varepsilon)$ and apply Lemma \ref{concatenation of selections} to $g_1=G_1|_{[a,c]}$ and $g_2=G_2|_{[c,t_1]}$ to produce a differentiable selection on $[a,t_1]$, which contradicts the choice of $t_0$.
\end{proof}

\begin{lemma}\label{local to global on closed interval}
Let $n,q\in\pinteger$, $I\subset\real$ be an interval and $f:I\rightarrow\qspace{q}{\real^n}$ be affinely approximable on $I$.
Suppose that for all $x\in I$ there exists an interval $J\subset I$ open relative to $I$ with $x\in J$ and $g:J\rightarrow (\real^n)^q$ a differentiable selection for $f|_J$.
Then, there exists $h:I\rightarrow(\real^n)^q$ a differentiable selection for $f$.
\end{lemma}
\begin{proof}
If $I=[a,b]$ with $a<b$ is a closed interval the we apply Lemma \ref{local to global selection} to $f|_{[a,b)}$ and $f|_{(a,b]}$, choose a point $c\in(a,b)$ and apply Lemma \ref{concatenation of selections} to obtain the result.
Similarly if $I=(a,b)$ is an open interval.
When $I=[a,b)$ or $I=(a,b]$ is a half-closed interval, then it is already proved on Lemma \ref{local to global selection}
\end{proof}

\begin{remark}
Observe that in the three Lemmas above we use strongly that the domain is one-dimensional.
Although one could conceive of a higher dimensional version of Lemma \ref{concatenation of selections}, the one-dimensional condition is particularly unavoidable in the proof of Lemma \ref{local to global selection}.
\end{remark}

\begin{remark}
If we start with a multiple-valued function $f$ satisfying $f(a)=f(b)$ it is not possible to guarantee that the selection $g$ constructed above will also have $g(a)=g(b)$.
That is, the construction does not extend to functions defined on $S^1$.
\end{remark}

The following Lemma is a version of the intermediate value property for derivatives of real valued functions, which we will use to derive a weaker version for $\real^n$-valued functions afterwards.

\begin{lemma}\label{intermediate value property}
If $a,r\in\real$, $r>0$ and $f:[a,a+r)\rightarrow\real$ is differentiable on $[a,a+r)$ and
\begin{equation*}
Y=\bigcap_{\varepsilon>0}\textup{clos}\{Df(x):a<x<a+\varepsilon\},
\end{equation*}
then $D^+f(a)\in Y$ and $Y$ is connected.
\end{lemma}
\begin{proof}
By the Mean Value Theorem we have that for all $\varepsilon>\delta>0$ there exists $c_\delta\in(a,a+\delta)$ such that 
\begin{equation*}
Df(c_\delta)=\delta^{-1}(f(a+\delta)-f(a)).
\end{equation*}
By letting $\delta\rightarrow 0^+$ we see that $D^+f(a)\in \textup{clos}\{Df(x):a<x<a+\varepsilon\}$ for all $\varepsilon>0$.
Hence $D^+f(a)\in Y$.

Now, if $Y$ consists of a single point then it is trivially connected.
Suppose that $\cardinality Y\geq 2$.
There exists $y_1,y_2\in Y$ and $z\in\real$ such that $y_1<z<y_2$.
Let $\varepsilon>0$ be arbitrary and put $\delta=\min\{\frac{|z-y_1|}{2},\frac{|z-y_2|}{2}\}$.
Since $y_1,y_2\in Y$ we can find $x_1,x_2\in(a,a+\varepsilon)$ with
\begin{equation*}
|y_i-Df(x_i)|<\delta \text{ for each }i=1,2
\end{equation*}
and in particular $Df(x_1)<z<Df(x_2)$.
Without loss of generality we may assume that $x_1<x_2$ and define $g:[x_1,x_2]\rightarrow\real$ as $g(x)=f(x)-zx$.
It follows that $g$ is continuous and therefore it attains a minimum in $[x_1,x_2]$ denoted by $g(x_0)$.

If $x_0=x_1$, then $D^+g(x_0)\geq 0$ and $Df(x_1)\geq z$, which is a contradiction.
Similarly for $x_0=x_2$.
We conclude that $x_0\in(x_1,x_2)$ and $Dg(x_0)=0$.
That is, $Df(x_0)=z$ and $z\in\{Df(x):a<x<a+\varepsilon\}$.
Since $\varepsilon$ was arbitrary we conclude that $z\in Y$, which implies that $Y$ is connected.
\end{proof}

\begin{lemma}\label{darboux lemma}
Let $N\in\pinteger$, $a,r\in\real$, $r>0$ and $f:[a,a+r)\rightarrow\real^N$ be differentiable on $[a,a+r)$.
If $Z\subset \real$ is a finite set satisfying
\begin{equation*}
\lim_{\varepsilon\rightarrow 0^+}\min\{|\langle e_i, Df(a+\varepsilon)\rangle-z|:z\in Z\}=0,
\end{equation*}
for all $i=1,\ldots,N$, where $e_i$ are the coordinate vectors in $\real^N$, then
\begin{equation*}
\lim_{\varepsilon\rightarrow 0^+}Df(a+\varepsilon)=D^+f(a).
\end{equation*}
\end{lemma}
\begin{proof}
For each $i=1,\ldots,N$ we put $f_i(x)=\langle e_i,f(x)\rangle$ and note that the condition on $Z$ implies trivially that 
\begin{equation*}
Y_i=\bigcap_{\varepsilon>0}\textup{clos}\{Df_i(x):a<x<a+\varepsilon\}\subset Z.
\end{equation*}
It follows from Lemma \ref{intermediate value property} that $Y_i$ is connected and contains $D^+f_i(a)$.
Since $Z$ is finite we conclude that $Y_i=\{D^+f_i(a)\}$ for each $i=1,\ldots, N$, which concludes the proof.
%
\end{proof}

We use this Lemma to prove that if a multiple-valued function has continuous derivative $Df$, then any differentiable selection must in fact be of class $1$.

\begin{proposition}\label{improvement from differentiable to class 1}
Let $n,q\in\pinteger$, $I\subset\real$ be an interval, $f:I\rightarrow\qspace{q}{\real^n}$ be a multiple-valued function and $g:I\rightarrow(\real^n)^q$ be a differentiable selection for $f$ on $I$.
If $Df$ is continuous then $g$ is of class $1$.
\end{proposition}
\begin{proof}
Fix $a\in I$ and for each $i\in\{1,\ldots,n\}$ denote by $e_i$ the canonical coordinate vector in $\real^n$.
For each $\alpha=1,\ldots,q$ we will abuse notation to identify $Dg_\alpha(a)\in\homspace{\real}{\real^n}$ with $Dg_\alpha(a)\in\homspace{\real^n}{\real}$ via the canonical inner product of Euclidean space.
That is,
\begin{equation*}
\langle v,Dg_\alpha(a)\rangle=\bflangle v,\langle 1, Dg_\alpha(a)\rangle \bfrangle_{\real^n}, \text{ for all }v\in\real^n.
\end{equation*}
We know from Lemma \ref{differentiability is preserved} that $Dg$ defines a selection for $Df$.
Since $Df$ is continuous, we have
\begin{equation*}
\lim_{x\rightarrow a,x\in I} \sum_{\alpha=1}^q\di{\langle e_i,Dg_\alpha(x)\rangle}=\sum_{\alpha=1}^q\di{\langle e_i,Dg_\alpha(a)\rangle}.
\end{equation*}
Therefore, if we define $Z=\{\langle e_i,Dg_\alpha(a)\rangle\in\real:\alpha=1,\ldots,q, i=1,\ldots,n\}$, we obtain that
\begin{equation*}
\lim_{\varepsilon\rightarrow 0^+}\min\{|\langle e_i, Dg_\alpha(a\pm\varepsilon)\rangle-z|:z\in Z\}=0
\end{equation*}
for each $\alpha=1,\ldots, q$ and $i=1,\ldots, n$.
If $a$ is an endpoint of $I$, it follows from Lemma \ref{darboux lemma} that
\begin{equation*}
\lim_{\varepsilon\rightarrow 0^+}Dg_\alpha(x\pm\varepsilon)=D^\pm g_\alpha(a).
\end{equation*}
That is, $Dg_\alpha$ is continuous at $a$ whether $a$ is a left or right endpoint of $I$.
If $a$ is an interior point, then we apply Lemma \ref{darboux lemma} to both sides to obtain
\begin{equation*}
\lim_{\varepsilon\rightarrow 0^+}Dg_\alpha(x-\varepsilon)=D^- g_\alpha(a)=D^+ g_\alpha(a)=\lim_{\varepsilon\rightarrow 0^+}Dg_\alpha(x+\varepsilon),
\end{equation*}
which concludes the proof.
\end{proof}

Unlike \cite{delellis-grisanti-tilli2004}*{Theorem 4.2} we are not making any assumption on the continuity of $Af$, only on $Df$.
However, we further assume that a differentiable selection is already given.

\begin{definition}\label{definition of splitting}
Let $n,q\in\pinteger$, $A\subset\real$, $f:A\rightarrow\qspace{q}{\real^n}$ be a multiple-valued function and $I\subset A$ be an interval.
We say that $f$ \textit{splits} on $I$ if there exists $k\in\pinteger$ with $k\geq 2$, $q_1,\ldots,q_k\in\pinteger$ with $q=q_1+\ldots+q_k$ and continuous functions $g_i:I\rightarrow\qspace{q_i}{\real^n}$ for each $i=1,\ldots,k$ such that
\begin{enumerate}[(i)]
\item $f(t)=\sum_{i=1}^kg_i(t)$ for all $t\in I$ and
\item for each $t\in I$, $\support{g_i(t)}\cap\support{g_j(t)}\neq\emptyset$ implies $i=j$.
\end{enumerate}

Given $a\in I$ we say that $g_1,\ldots,g_k$ is a \textit{pointed split} at $(I,a)$ if in addition $f$ is affinely approximable at $a$, each $g_i$ is affinely approximable at $a$ and condition $(ii)$ only holds for $t\in I\setminus\{a\}$.
\end{definition}

\begin{figure}[H]
\captionsetup{justification=centering}
\centering
\begin{subfigure}{.5\textwidth}
  \centering
  \begin{tikzpicture}[yscale=0.4,xscale=0.2]

\draw[red,thick] (-8,2) .. controls (-6,0.5) and (-2,0) .. (0,0) .. controls (2,0) and (6,0.5) .. (8,2);
\draw[red,thick] (-8,-2) .. controls (-6,-0.5) and (-2,0) .. (0,0) node (v1) {} .. controls (2,0) and (6,-0.5) .. (8,-2);

\draw[blue,thick] (-8,-4) .. controls (-6,-2) and (-2.5,-1) .. (0,-2.5) .. controls (2.5,-4) and (6,-6) .. (8,-4);

\draw[thick] (-8,-6) .. controls (-6,-6) and (-3,-6) .. (0,-8) .. controls (3,-10) and (6,-10) .. (8,-10);
\draw[thick] (-8,-10) .. controls (-6,-10) and (-3,-10) .. (0,-8) .. controls (3,-6) and (6,-6) .. (8,-6);

\draw[->, very thick] (-10,-12) -- (10,-12);
\draw[dashed] (-8,-12) -- (-8,3);
\draw[dashed] (8,-12) -- (8,3);

\draw[very thick] (0,-12.2) -- (0,-11.8);
\draw[dotted, thick] (0,0) -- (0,-12);

\node [label={[label distance = 0.1 cm]-90:$x$}] at (0,-12) {};

\node [label={[label distance = 0.4 cm]90:graph$(f)$, $q=5$}] at (0,2) {};
\node [label={[label distance = 0 cm]90:card supp$f(x)=3$}] at (0,2) {};

\node [label={[label distance = -0.35 cm]0:{\color{red}$g_1$, $q_1=2$}}] at (6,0) {};
\node [label={[label distance = -0.35 cm]0:{\color{blue}$g_2$, $q_2=1$}}] at (6,-3.7) {};
\node [label={[label distance = -0.35 cm]0:$g_3$, $q_3=2$}] at (6,-8) {};

\node [label={[label distance = -0.15 cm]90:{\color{red}$1$}}] at (4.5,0.5) {};
\node [label={[label distance = -0.15 cm]-90:{\color{red}$1$}}] at (4.5,-0.5) {};

\node [label={[label distance = -0.15 cm]-90:{\color{blue}$1$}}] at (4,-4.5) {};

\node [label={[label distance = -0.15 cm]-90:$1$}] at (5.5,-6) {};
\node [label={[label distance = -0.15 cm]90:$1$}] at (5.5,-10) {};

\node [label={[label distance = 2 cm]120:}] at (-8,0) {};
\end{tikzpicture}
  \caption*{Local splitting near $f(x)$\\ with disjoint support.}
  \label{fig:splitting 3}
\end{subfigure}%
\begin{subfigure}{.5\textwidth}
  \centering
  \begin{tikzpicture}[yscale=0.4,xscale=0.2]

\draw[red, thick] (-8,-2) .. controls (-7,-1) and (-2.5,-1.5) .. (0,0) .. controls (2,1.5) and (4,2) .. (8,2);
\draw[red, thick] (-8,-3) .. controls (-3,-3) and (-2.5,-1.5) .. (0,0) .. controls (2,1.5) and (7,1.5) .. (8,1);

\draw[thick] (-8,1.5) .. controls (-5.5,-0.5) and (-2,2) .. (0,0) .. controls (3,-2) and (4,0.5) .. (8,-3);
\draw[thick] (-8,1.5) .. controls (-5.5,-0.5) and (-2,2) .. (0,0) .. controls (3,-2) and (5,-3.5) .. (8,-1.5);

\draw[blue, thick] (-8,0) -- (8,0);

\node [label={[label distance = 0.8 cm]90:graph$(f)$, $q=5$}] at (0,5) {};
\node [label={[label distance = 0.4 cm]90:card supp $f(x)=1$}] at (0,5) {};
\node [label={[label distance = 0 cm]90:card supp $Df(x)=3$}] at (0,5) {};

\node [label={[label distance = 0 cm]10:{\color{red}$g_1$, $q_1=2$}}] at (7,1.5) {};
\node [label={[label distance = 0 cm]10:{\color{blue}$g_2$, $q_2=1$}}] at (7,-0.5) {};
\node [label={[label distance = 0 cm]30:{$g_3$, $q_3=2$}}] at (7,-2.5) {};

\node [label={[label distance = -0.1 cm]90:$2$}] at (-5.5,0.5) {};
\node [label={[label distance = -0.2 cm]-90:$1$}] at (4,-1) {};
\node [label={[label distance = -0.2 cm]-90:$1$}] at (4.5,-2.5) {};

\node [label={[label distance = -0.15 cm]-90:{\color{blue}$1$}}] at (6,0) {};

\node [label={[label distance = -0.15 cm]-90:{\color{red}$1$}}] at (6,1.5) {};
\node [label={[label distance = -0.15 cm]90:{\color{red}$1$}}] at (6,2) {};

\draw[->, very thick] (-10,-8) -- (10,-8);
\draw[dashed] (-8,-8) -- (-8,3);
\draw[dashed] (8,-8) -- (8,3);
\draw[very thick] (0,-8.2) -- (0,-7.8);
\draw[dotted, thick] (0,0) -- (0,-8);
\node [label={[label distance = 0.1 cm]-90:$x$}] at (0,-8) {};

\node [label={[label distance = 2 cm]120:}] at (-8,0) {};

\end{tikzpicture}
  \caption*{Local pointed splitting near $f(x)$\\ concentrating at a single point.}
  \label{fig:pointed splitting 3}
\end{subfigure}
\label{fig:test}
\end{figure}

We note that the above definition is similar to the notion of decomposition in \cite{delellis-espadaro2011}*{Definition 1.1} but we require stronger separation of the decomposing functions.
Although it is not clear whether continuity of $Df$ is preserved under splittings, we are nonetheless able to prove that affine approximability is preserved.

\begin{lemma}\label{splitting preserves affine approximation}
Let $n,q\in\pinteger$, $A\subset\real$, $f:A\rightarrow\qspace{q}{\real^n}$ be a multiple-valued function, $I\subset A$ be an interval and $a\in I$.
Suppose that $f$ splits on $I$ and the splitting is defined by the functions $g_1,\ldots,g_k$ for some $k\geq 2$.

If $f$ is affinely approximable at $a$ relative to $I$, then each $g_i$ is affinely approximable at $a$ relative to $I$ and
\begin{equation*}
\tfa{f}(a)=\sum_{i=1}^k \tfa{g_i}(a).
\end{equation*}
\end{lemma}
\begin{proof}
For each $i=1,\ldots,k$, let $h_{i,1},\ldots,h_{i,q_{i}}:I\rightarrow \real^n$ be an arbitrary selection for $g_i:I\rightarrow \qspace{q_i}{\real^n}$, $Q_i=\sum_{j=1}^{i}q_j$ and $Q_0=0$. 
We define for each $\alpha=1,\ldots, q$
\begin{equation*}
f_\alpha(t)=h_{i,\alpha-Q_{i-1}}(t), \text{ when } Q_{i-1}<\alpha\leq Q_i.
\end{equation*}
Therefore $f_1,\ldots,f_q$ defines a selection for $f$ such that
\begin{equation*}
g_i(t)=\sum_{\alpha=Q_{i-1}+1}^{Q_i}\di{f_\alpha(t)}
\end{equation*}
for each $i=1,\ldots, k$.

We may write $\tfa{f}(a)=\sum_{\alpha=1}^q\di{(f_\alpha(a),L_\alpha)}$, for some $L_1,\ldots,L_q\in\homspace{\real}{\real^n}$, and for each $t\in I$ let $\sigma_t\in\perm{q}$ be the permutation such that
\begin{equation*}
\fmetric(f(t),\pfa{a}{f}(t))=\sum_{\alpha=1}^q\|f_{\sigma_t(\alpha)}(t)-f_\alpha(a)-\langle t-a,L_\alpha\rangle\|.
\end{equation*}
\begin{claim}
There exists $\varepsilon>0$ such that for all $t\in I\cap\openball{a}{\varepsilon}$, $i=1,\ldots,k$ and $\alpha\in\pinteger\cap(Q_{i-1},Q_i]$ we have $\sigma_t(\alpha)\in\pinteger\cap(Q_{i-1},Q_i]$.
\end{claim}
Suppose false, that is, we can find a sequence $t_m\in I$ converging to $a$,  $i_m\in\{1,\ldots,k\}$ and $\alpha_m\in\pinteger\cap(Q_{i-1},Q_i]$ such that the statement is false.
By possibly taking a subsequence, we may assume that $i_m$, $\alpha_m$ and $\sigma_{t_m}(\alpha_m)$ are constants.
Without loss of generality we may take $i_m\equiv 1$, $\alpha_m\equiv 1$ and $\sigma_{t_m}(\alpha_m)=q_1+1$.

By letting $t_m$ tend to $a$, it follows from the choice of $\sigma_{t_m}$ that $f_{q_1+1}(t_m)\rightarrow f_1(a)$.
On the other hand, $f_{q_1+1}(t_m)\in\support g_2(t_m) $ and since $g_2$ is continuous, we have that $f_1(a)\in\support g_2(a)$.
That is, $\support g_1(a)\cap\support g_2(a)\neq\emptyset$, which contradicts condition \ref{definition of splitting}(ii).

Now, if we put 
\begin{equation*}
\Phi_i=\sum_{\alpha=Q_{i-1}+1}^{Q_i}\di{(f_\alpha(a),L_\alpha)}
\end{equation*}
for each $i=1,\ldots,k$, it follows from the claim that
\begin{equation*}
\fmetric(f(t),\pfa{a}{f}(t))=\sum_{i=1}^k\fmetric(g_i(t),\up{P}_{a,\Phi_i}(t))
\end{equation*}
for all $t$ sufficiently close to $a$, which concludes the proof.
\end{proof}

\begin{lemma}\label{pointed splitting preserves affine approximation}
Let $n,q\in\pinteger$, $A\subset\real$, $f:A\rightarrow\qspace{q}{\real^n}$ be a multiple-valued function, $I\subset A$ be an interval and $a\in I$.
Suppose that $f$ admits a pointed split on $(I,a)$ and the splitting is defined by the functions $g_1,\ldots,g_k$ for some $k\geq 2$.

If $f$ is affinely approximable at $b\in I\setminus\{a\}$ relative to $I$, then each $g_i$ is affinely approximable at $b$ relative to $I$ and
\begin{equation*}
\tfa{f}(b)=\sum_{i=1}^k \tfa{g_i}(b).
\end{equation*}
\end{lemma}
\begin{proof}
Simply note that if $f$ has a pointed split on $(I,a)$ and $b\in I\setminus\{a\}$, then we can find an interval $J\subset I\setminus\{a\}$ such that $b\in J$ and $g_1|_{J},\ldots,g_k|_{J}$ define a spliting of $f$ on $J$.
The result follows from Lemma \ref{splitting preserves affine approximation}.
\end{proof}

Now, let us concern ourselves with the existence of local splittings.
First we consider local splittings near points where $f$ does not concentrate on a single point and only assuming continuity of $f$.

\begin{lemma}\label{existence of splitting}
Let $n,q\in\pinteger$, $I\subset\real$ be an interval, $f:I\rightarrow\qspace{q}{\real^n}$ be a continuous multiple-valued function and $a\in I$.
Suppose $\cardinality{\support{f(a)}}\geq 2$, then there exists an interval $J$ relatively open in $I$ such that $a\in J$ and $f$ splits on $J$.
Furthermore, we can make it so that the number of splitting functions is $k=\cardinality{\support{f(a)}}$ and that each split function $g_i$ satisfies $\cardinality{\support{g_i(a)}}=1$.
\end{lemma}
\begin{proof}
Let $f_{\alpha}:I\rightarrow\real^n$ be an arbitrary selection for $f$ with $\alpha=1,\ldots,q$.
Put
\begin{equation*}
\eta=\inf\{\|y-y'\|:y,y'\in\support{f(a)},\,y\neq y'\}, k=\cardinality{\support{f(a)}}
\end{equation*}
and for each $i=1,\ldots,k$ we enumerate $y_i\in\support{f(a)}$ and write $q_i=\density{0}{f(a)}{y_i}$.
For each $t\in I$ there exists $\sigma_t\in\perm{q}$ such that
\begin{equation*}
\fmetric(f(t),f(a))=\sum_{\alpha=1}^q\|f_{\sigma_t(\alpha)}(t)-f_\alpha(a)\|
\end{equation*}
For each $i=1,\ldots,k$ we define $\mathcal{A}_i=\{\alpha\in\{1,\ldots,q\}:f_\alpha(a)=y_i\}$ and
\begin{equation*}
g_i(t)=\sum_{\alpha\in\mathcal{A}_i}\di{f_{\sigma_t(\alpha)}(t)}.
\end{equation*}
It is clear that $\cardinality{\mathcal{A}_i}=q_i$, $g_i(t)\in\qspace{q_i}{\real^n}$, $g_i(a)=q_i\di{y_i}$ and
\begin{equation*}
f(t)=\sum_{i=1}^kg_i(t).
\end{equation*}
That is, Definition \ref{definition of splitting}(i) holds.

Since $f$ is continuous, there exists $\delta>0$ such that
\begin{equation*}
\fmetric(f(t),f(a))<\frac{\eta}{8} \text{ for all }t\in (a-\delta,a+\delta)\cap I.
\end{equation*}
We fix $J=(a-\delta,a+\delta)\cap I$.

To verify Definition \ref{definition of splitting}(ii) we first note that 
\begin{equation*}
\begin{aligned}
\frac{\eta}{8} & >\fmetric(f(t),f(a))\\
               & = \sum_{\alpha=1}^q\|f_{\sigma_t(\alpha)}(t)-f_\alpha(a)\|\\
							 & = \sum_{i=1}^k\sum_{\alpha\in\mathcal{A}_i}\|f_{\sigma_t(\alpha)}(t)-y_i\|.
\end{aligned}
\end{equation*}
Hence, for all $t\in J$, $i=1,\ldots,k$ and $\alpha\in\mathcal{A}_i$ we have
\begin{equation*}
f_{\sigma_t(\alpha)}(t)\in\openball{y_i}{\frac{\eta}{8}}.
\end{equation*}
Furthermore, the collection $\{\openball{y_i}{\frac{\eta}{8}}:i=1,\ldots,k\}$ is pairwise disjoint.
Indeed, let $i\neq j$, $z\in\openball{y_i}{\frac{\eta}{8}}$ and $z'\in\openball{y_j}{\frac{\eta}{8}}$ then
\begin{equation*}
\begin{aligned}
\|z-z'\| & = \|z-y_i-z'+y_j+y_i-y_j\|\\
       & \geq \|y_i-y_j\| - \|z-y_i\| - \|z'-y_j\|\\
			 & \geq \frac{3\eta}{4}
\end{aligned}
\end{equation*}
In other words, for all $i\neq j$ and $t\in J$ we have
\begin{equation*}
\support{g_i(t)}\cap\support{g_j(t)}=\emptyset.
\end{equation*}

It also follows from the above that
\begin{equation*}
f(t)\restrict \openball{y_i}{\frac{\eta}{8}}=g_i(t).
\end{equation*}

It remains to prove that each $g_i$ is continuous on $J$.
Fix $t_0\in J$, take $s\in J$ and apply Lemma \ref{lemma for local fmetric}(3) with $S=f(a)$, $\varepsilon=\frac{\eta}{6}$, $T_1=f(t_0)$ and $T_2=f(s)$.
From the choice of $J$ we already know that $\fmetric(T_1,S)<\varepsilon$.
Given an arbitrary $0<\varepsilon'<\varepsilon$ we may find $\delta'>0$ such that $|s-t_0|<\delta'$ and $s\in J$ implies $\fmetric(f(t_0),f(s))<\varepsilon'$, that is, that is, $\fmetric(T_2,S)<\varepsilon$.
Thus,
\begin{equation*}
\fmetric_q(f(t_0),f(s))=\sum_{i=1}^k\fmetric_{q_i}(g_i(t_0),g_i(s)),
\end{equation*}
which concludes the proof.
\end{proof}

\begin{remark}
Note that the proofs of the three Lemmas above do not rely on the dimension of the domain.
\end{remark}

In the next case we construct local pointed splittings near points $a$ where $f$ is affinely approximable at $a$, $f(a)$ concentrates in a single point but $f$ is not strongly affinely approximable at $a$.

\begin{lemma}\label{existence of pointed splitting}
Let $n,q\in\pinteger$, $I\subset\real$ be an interval, $f:I\rightarrow\qspace{q}{\real^n}$ be a continuous multiple-valued function and $a\in I$.
Suppose that $f$ is affinely approximable at $a$, $\cardinality{\support{f(a)}}=1$ and $\cardinality{\support{Df(a)}}\geq 2$.
Then there exists an interval $J\subset I$ such that $f$ admits a pointed split on $(J,a)$.
Furthermore, we can make it so that the number of splitting functions is $k=\cardinality{\support{Df(a)}}$ and that each pointed split function $g_i$ satisfies $\cardinality{\support{Dg_i(a)}}=1$.
\end{lemma}
\begin{proof}
Let $f_\alpha:I\rightarrow\real^n$ be an arbitrary selection for $f$ with $\alpha=1,\ldots,q$, $y_0\in\real^n$ with $\support f(a)=\{y_0\}$ and $L_1,\ldots,L_q\in\homomorphism{\real}{\real^n}$ with $Df(a)=\sum_{\alpha=1}^q\di{L_\alpha}$.
Put 
\begin{equation*}
\eta=\inf\{\|M-M'\|_{\odot^1}:M,M'\in\support{Df(a)},\,M\neq M'\}, k=\cardinality{\support{Df(a)}},
\end{equation*}
enumerate $\support{Df(a)}=\{M_1,\ldots,M_k\}$ and write $q_i=\density{0}{Df(a)}{M_i}$ for each $i=1,\ldots,k$.
For every $t\in I$ there exists $\sigma_t\in\perm{q}$ such that
\begin{equation*}
\fmetric(f(t),\pfa{a}{f}(t))=\sum_{\alpha=1}^q\|f_{\sigma_t(\alpha)}(t)-y_0-\langle t-a,L_\alpha\rangle\|
\end{equation*}
Now, for each $i=1,\ldots,k$ we define $\mathcal{A}_i=\{\alpha\in\{1,\ldots,q\}:L_\alpha=M_i\}$ and
\begin{equation*}
g_i(t)=\sum_{\alpha\in\mathcal{A}_i}\di{f_{\sigma_t(\alpha)}(t)}.
\end{equation*}
It is clear that $\cardinality{\mathcal{A}_i}=q_i$, $g_i(t)\in\qspace{q_i}{\real^n}$ and
\begin{equation*}
f(t)=\sum_{i=1}^kg_i(t).
\end{equation*}
That is, Definition \ref{definition of splitting}(i) holds.
Furthermore, by choice of $\sigma_t$, we have 
\begin{equation*}
\begin{aligned}
\fmetric(f(t),\pfa{a}{f}(t)) & =\sum_{i=1}^k\sum_{\alpha\in\mathcal{A}_i}\|f_{\sigma_t(\alpha)}(t)-y_0-\langle t-a,M_i\rangle\|
\end{aligned}
\end{equation*}
Hence, for each $i=1,\ldots,k$ and $\alpha\in\mathcal{A}_i$ the function $t\mapsto f_{\sigma_t(\alpha)}(t)$ is differentiable at $a$ with derivative equal to $M_i$ and $g_i$ is affinely approximable at $a$ with 
\begin{equation*}
\tfa{g_i}(a)=q_i\di{(y_0,M_i)} \text{ for each } i=1,\ldots,k.
\end{equation*}

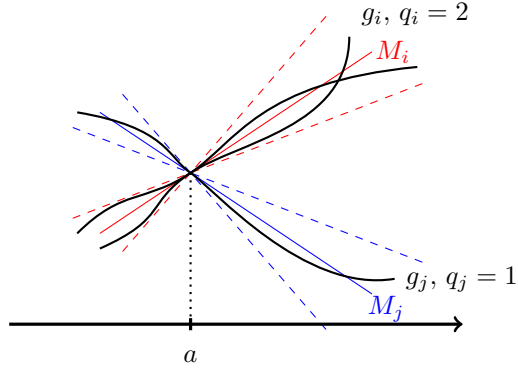
\begin{figure}[H]
\captionsetup{justification=centering}
\centering
\begin{tikzpicture}[yscale=0.4,xscale=0.6]

\draw[red] (-2,-2) -- (4,4);
\draw[red,dashed](30:-3)--(0:0) -- (30:6);
\draw[red,dashed](60:-3)--(0:0) -- (60:6);
\node [label={[label distance = -0.2 cm]	0:{\color{red}$M_i$}}] at (4,4) {};

\draw[blue] (-2,2) -- (4,-4);
\draw[blue,dashed](-30:-3)--(0:0) -- (-30:6);
\draw[blue,dashed](-60:-3)--(0:0) -- (-60:6);
\node [label={[label distance = -0.35 cm]-30:{\color{blue}$M_j$}}] at (4,-4) {};

\draw[thick] (-2,-2.5) .. controls (-0.5,-1.5) and (-1,-1) .. (0,0) .. controls (1,1) and (3.5,1.5) .. (3.5,4.5);
\draw[thick] (-2.5,-2) .. controls (-1.5,-0.5) and (-1,-1) .. (0,0) .. controls (1,1) and (1.5,3) .. (5,3.5);
\node [label={[label distance = -0.1 cm]30:{$g_i$, $q_i=2$}}] at (3.5,4.5) {};

\draw[thick] (-2.5,2) .. controls (-0.5,1.5) and (-0.5,0.5) .. (0,0) .. controls (0.5,-0.5) and (2.5,-4) .. (4.5,-3.5);
\node [label={[label distance = -0.1 cm]0:{$g_j$, $q_j=1$}}] at (4.5,-3.5) {};

\draw[->, very thick] (-4,-5) -- (6,-5);
\draw[very thick] (0,-5.2) -- (0,-4.8);
\draw[dotted, thick] (0,0) -- (0,-5);
\node [label={[label distance = 0.1 cm]-90:$a$}] at (0,-5) {};
\end{tikzpicture}
\label{fig:pointed split diagram 2}
\end{figure}

Since $f$ is affinely approximable at $a$, there exists $\delta>0$ such that
\begin{equation*}
\fmetric(f(t),\pfa{a}{f}(t))<\frac{\eta}{8}|t-a| \text{ for all }t\in (a-\delta,a+\delta)\cap I.
\end{equation*}
We fix $J=(a-\delta,a+\delta)\cap I$.

To verify Definition \ref{definition of splitting}(ii) away from $a$, let $t\in J\setminus\{a\}$ and note that
\begin{equation*}
\|f_{\sigma_t(\alpha)}(t)-y_0-\langle t-a,M_i\rangle\|<\frac{\eta}{8}|t-a|\text{ for all }i=1,\ldots,k\text{ and }\alpha\in\mathcal{A}_i.
\end{equation*}
That is,
\begin{equation*}
\support{g_i(t)}\subset\openball{y_0+\langle t-a, M_i\rangle}{\frac{\eta}{8}|t-a|}\text{ for all }i=1,\ldots,k.
\end{equation*}
We also observe that for all $t\in J\setminus\{a\}$ the collection of slightly larger balls $\{\openball{y_0+\langle t-a,M_i\rangle}{\frac{\eta}{6}|t-a|}:i=1,\ldots,k\}$ is pairwise disjoint.
Indeed, let $i\neq j$, $z\in\openball{y_0+\langle t-a, M_i\rangle}{\frac{\eta}{6}|t-a|}$ and $z'\in\openball{y_0+\langle t-a, M_j\rangle}{\frac{\eta}{6}|t-a|}$ then
\begin{equation*}
\begin{aligned}
\|z-z'\| & = \|z-y_0-\langle t-a, M_i\rangle-z'+y_0+\langle t-a, M_j\rangle+\langle t-a, M_i- M_j\rangle\|\\
       & \geq \|M_i-M_j\|_{\odot^1}|t-a| - \|z-y_0-\langle t-a, M_i\rangle\| - \|z'-y_0-\langle t-a, M_j\rangle\|\\
			 & \geq \eta|t-a| - \frac{\eta}{3}|t-a|\\
			 & = \frac{2\eta}{3}|t-a|,
\end{aligned}
\end{equation*}
where in the second line we used the fact that the domain is one-dimensional to obtain
\begin{equation*}
\|\langle t-a,M_i-M_j\rangle\|_{\real^n}=\|M_i-M_j\|_{\odot^1}|t-a|.
\end{equation*}
Thus, $\support{g_i(t)}\cap\support{g_j(t)}=\emptyset$ for all $i\neq j$ and $t\in J\setminus\{a\}$, which proves Definition \ref{definition of splitting}(ii) away from $a$.
It further follows from the above that
\begin{equation*}
f(t)\restrict\openball{y_0+\langle t-a,M_i\rangle}{\frac{\eta}{8}|t-a|}=g_i(t)
\end{equation*}
for all $i=1,\ldots,k$ and $t\in J\setminus\{a\}$.

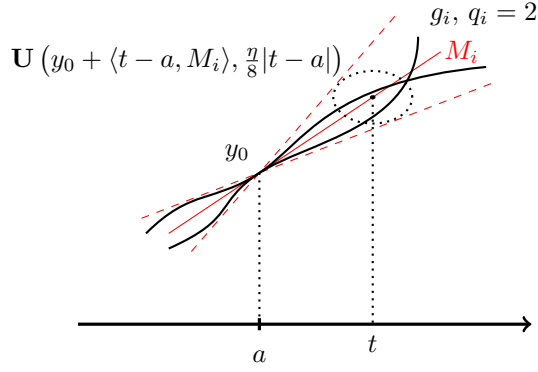
\begin{figure}[H]
\captionsetup{justification=centering}
\centering
\begin{tikzpicture}[yscale=0.4,xscale=0.6]

\draw[red] (-2,-2) -- (4,4);
\draw[red,dashed](30:-3)--(0:0) -- (30:6);
\draw[red,dashed](60:-3)--(0:0) -- (60:6);
\node [label={[label distance = -0.2 cm]0:{\color{red}$M_i$}}] at (4,4) {};

\draw[thick] (-2,-2.5) .. controls (-0.5,-1.5) and (-1,-1) .. (0,0) .. controls (1,1) and (3.5,1.5) .. (3.5,4.5);
\draw[thick] (-2.5,-2) .. controls (-1.5,-0.5) and (-1,-1) .. (0,0) .. controls (1,1) and (1.5,3) .. (5,3.5);
\node [label={[label distance = -0.1 cm]30:{$g_i$, $q_i=2$}}] at (3.5,4.5) {};

\draw[->, very thick] (-4,-5) -- (6,-5);
\draw[very thick] (0,-5.2) -- (0,-4.8);
\draw[dotted, thick] (0,0) -- (0,-5);
\node [label={[label distance = 0.1 cm]-90:$a$}] at (0,-5) {};

\draw[dotted, thick]  (2.5,2.5) node (v1) {} circle (0.87);
\node [label={[label distance = -0.1 cm]100:$y_0$}] at (0,0) {};

\draw[dotted, thick] (2.5,2.5) -- (2.5,-5);
\node [label={[label distance = -0.1 cm]-90:$t$}] at (2.5,-5) {};
\node [label={[label distance = 0.15 cm]150:$\openball{y_0+\langle t-a,M_i\rangle}{\frac{\eta}{8}|t-a|}$}] at (2.5,2.5) {};
\draw[fill] (2.5,2.5) circle (0.05);
\end{tikzpicture}
\label{fig:pointed split diagram 1}
\end{figure}

It remains to prove that each $g_i$ is continuous on $J$.
It follows trivially from the construction that $g_i$ is continuous at $a$.
Now fix $t_0\in J\setminus\{a\}$ and $\varepsilon=\frac{\eta}{6}|t_0-a|$.
If $0<\varepsilon'<\varepsilon$ is arbitrary, we may find $\delta'>0$ so that $|s-t_0|<\delta'$ and $s\in J\setminus\{a\}$ implies
\begin{equation*}
\left(\frac{\eta}{8}+\sup\{\|M_i\|_{\odot^1}:i=1,\ldots,k\}\right)|s-t_0|<\frac{\eta}{24}|t_0-a|
\end{equation*}
and 
\begin{equation*}
\fmetric(f(t_0),f(s))<\varepsilon'.
\end{equation*}
First we note that for each $i=1,\ldots,k$ and $|s-t_0|<\delta'$ we have
\begin{equation*}
\support{g_i(s)}\subset\openball{y_0+\langle t_0-a,M_i\rangle}{\frac{\eta}{6}|t_0-a|}.
\end{equation*}
Indeed, let $\alpha\in\mathcal{A}_i$ and compute
\begin{equation*}
\begin{aligned}
\|f_{\sigma_s(\alpha)}(s)-(y_0+\langle t_0-a,M_i\rangle)\| & \leq \|f_{\sigma_s(\alpha)}(s)-(y_0+\langle s-a,M_i\rangle)\|\\
&\quad + \|M_i\|_{\odot^1}|s-t_0|\\
                             & \leq \frac{\eta}{8}|s-a|+\|M_i\|_{\odot^1}|s-t_0|\\
														 & \leq \frac{\eta}{8}|t_0-a|+\left(\frac{\eta}{8}+\|M_i\|_{\odot^1}\right)|s-t_0|\\
														 & < \frac{\eta}{8}|t_0-a|+\frac{\eta}{24}|t_0-a|\\
												     & = \frac{\eta}{6}|t_0-a|.
\end{aligned}
\end{equation*}
Therefore, we have
\begin{equation*}
f(s)\restrict\openball{y_0+\langle t_0-a, M_i\rangle}{\frac{\eta}{6}|t_0-a|}=g_i(s)
\end{equation*}
for all $|s-t_0|<\delta'$ and $i=1,\ldots,k$.

Finally, we define 
\begin{equation*}
S=\sum_{\alpha=1}^q\di{y_0+\langle t_0-a, L_\alpha\rangle}=\sum_{i=1}^kq_i\di{y_0+\langle t_0-a,M_i\rangle},
\end{equation*}
$\varepsilon=\frac{\eta}{6}|t_0-a|$, $T_1=f(t_0)$ and $T_2=f(s)$.
Observe that
\begin{equation*}
\inf\{\|y-y'\|:y,y'\in\support{S},\,y\neq y'\}=\eta|t_0-a|,
\end{equation*}
$\fmetric(T_1,S)\leq\varepsilon$ and if $|s-t_0|<\delta'$ then $\fmetric(T_2,S)<\varepsilon$.
We apply Lemma \ref{lemma for local fmetric}(3) to $S$, $T_1$, $T_2$ and $\varepsilon$ and obtain
\begin{equation*}
\fmetric_q(f(t_0),f(s))=\sum_{i=1}^k\fmetric_{q_i}(g_i(t_0),g_i(s)),
\end{equation*}
which concludes the proof.
\end{proof}

It is not always possible to find a local pointed splitting near $q$-points where $f$ is strongly affinely approximable.
Nevertheless, as we will see, we are still able to find a local differentiable selection near such points.

\begin{theorem}\label{main theorem 1}
Let $n,q\in\pinteger$, $I\subset\real$ be an interval and $f:I\rightarrow\qspace{q}{\real^n}$ be a multiple-valued function.
If $f$ is affinely approximable on $I$, then $f$ admits a differentiable selection.
\end{theorem}
\begin{proof}
We will prove it by induction on $q$.
The case $q=1$ is trivial and there is nothing to prove.
Take $q\geq 2$ and suppose that the statement holds for all $1\leq q'<q$.

We define
\begin{equation*}
\begin{aligned}
A & = \{t\in I:\cardinality{\support{f(t)}}=1\},\\
B & = \{t\in A:\cardinality{\support{Df(t)}}\geq 2\} \text{ and }\\
C & = \{t\in A:\cardinality{\support{Df(t)}}=1\}.
\end{aligned}
\end{equation*}
Note that $A$ is closed relative to $I$, $B$ corresponds to the points in $A$ in which $f$ is not strongly affinely differentiable and $C$ corresponds to the points in $A$ where $f$ is strongly affinely approximable.

\begin{figure}[H]
\captionsetup{justification=centering}
\centering
\usetikzlibrary{arrows}
\begin{tikzpicture}[scale=0.5]

\draw[very thick] (-6,-2) -- (-4,-2) node (v1) {};
\draw[very thick] (-6,-2.2) -- (-6,-1.8);
\draw[very thick] (-4,-2.2) -- (-4,-1.8);

\draw (-6,0) -- (6,0);
\draw (-4,0) .. controls (-3.5,0) and (-2.5,1.5) .. (-2,1.5) .. controls (-1.5,1.5) and (-1,1) .. (-0.5,0) .. controls (0,-1) and (1,-1.5) .. (2,-1.5) .. controls (3,-1.5) and (4.5,0) .. (5,0);
\draw (-4,0) .. controls (-3.5,0) and (-3,-1.5) .. (-2.5,-1.5) .. controls (-2,-1.5) and (-1,-0.5) .. (-0.5,0) node (v3) {} .. controls (0,0.5) and (0.5,1) .. (1,1) .. controls (1.5,1) and (2,0) .. (2.5,0);
\draw (3,0) .. controls (3.5,0) and (3.5,1) .. (4,1) .. controls (4.5,1) and (4.5,0) .. (5,0);

\node [label={[label distance = -0.1 cm]90:{$3$}}] at (-5,0) {};

\node [label={[label distance = -0.1 cm]90:{$1$}}] at (-2,1.5) {};
\node [label={[label distance = -0.1 cm]90:{$1$}}] at (-2,0) {};
\node [label={[label distance = -0.1 cm]90:{$1$}}] at (-2.5,-1.5) {};

\node [label={[label distance = -0.1 cm]90:{$3$}}] at (-0.5,0) {};

\node [label={[label distance = -0.1 cm]30:{$2$}}] at (2.5,0) {};
\node [label={[label distance = -0.3 cm]120:{$1$}}] at (3,-1) {};

\node [label={[label distance = -0.1 cm]90:{$3$}}] at (5,0) {};

\draw[red,thick] (-3.95,-2) -- (-0.58,-2);
\draw[fill, blue] (-0.5,-2) circle (0.05);
\draw[red,thick] (-0.38,-2) -- (4.9,-2);
\draw[red,thick] (5.08,-2) -- (6,-2);
\draw[very thick] (5,-2.2) -- (5,-1.8);

\node [label={[label distance = 0.1 cm]-90:{\color{blue}$B$}}] (v2) at (-0.5,-2) {};
\node [label={[label distance = 0.1 cm]-90:{\color{red}$I\setminus A$}}] at (-2,-2) {};
\node [label={[label distance = 0.1 cm]-90:{\color{red}$I\setminus A$}}] at (2.5,-2) {};
\node [label={[label distance = 0.1 cm]-90:{\color{red}$I\setminus A$}}] at (6.2,-2) {};
\node [label={[label distance = 0.1 cm]-90:{$C$}}] (v4) at (5,-2) {};
\node [label={[label distance = 0.1 cm]-90:{$C$}}] at (-5,-2) {};
\draw (5,0) .. controls (5.5,0) and (5.5,-1) .. (6,-1.5);
\draw (5,0) node (v5) {} .. controls (5.5,0) and (5.5,0.5) .. (6,1);
\draw[blue, dotted] (v2) -- (v3);
\draw[dotted] (v4) -- (v5);
\end{tikzpicture}
\label{fig:domain}
\end{figure}

\begin{claim}
$C$ is closed relative to $I$.
\end{claim}
Since $C\subset A$ and $A$ is closed relative to $I$, it is sufficient to prove that $C$ is closed relative to $A$ or equivalently that $B=A\setminus C$ is open relative to $A$.
To that effect we take $b\in B$ and we will prove that $b$ is isolated in $A$.
It follows that
\begin{equation*}
f(b)=q\di{y_0} \text{ and } Df(b)=\sum_{\alpha=1}^q\di{M_\alpha}
\end{equation*}
for some $y_0\in\real^n$ and $M_1,\ldots,M_q\in\homspace{\real}{\real^n}$ with at least two distinct elements.
Since the domain is one-dimensional we have
\begin{equation*}
\|\langle x,L\rangle\|_{\real^n}=\|L\|_{\odot^1(\real,\real^n)}|x|
\end{equation*}
for all $x\in\real$ and $L\in\homspace{\real}{\real^n}$.
Now, let 
\begin{equation*}
\eta_0=\inf\{\|M-M'\|:M,M'\in\support{Df(b)}, M\neq M'\}>0
\end{equation*} 
 and $S=\pfa{b}{f}(x)$ for some $x\in I\setminus\{b\}$.
If $y,y'\in\support{S}$, then
\begin{equation*}
\begin{aligned}
y-y'& = y_0+\langle x-b,M\rangle-y_0-\langle x-b,M'\rangle\\
    & = \langle x-b, M-M'\rangle
\end{aligned}
\end{equation*}
for some $M,M'\in\support{Df(b)}$.
Therefore,
\begin{equation*}
\begin{aligned}
\eta & = \inf\{\|y-y'\|:y,y'\in \support{S},y\neq y'\}\\
     & = \inf\{\|M-M'\|:M,M'\in\support{Df(b)}, M\neq M'\}|x-b|\\
		 & = \eta_0|x-b|.
\end{aligned}
\end{equation*} 
We may find $\delta>0$ sufficiently small such that if $0<|x-b|<\delta$ and $x\in I$ then $\fmetric(f(x),S)<\frac{\eta}{3}$.
We apply Lemma \ref{lemma for local fmetric}(2) to $S$, $T=f(x)$ and $\varepsilon=\frac{\eta}{3}$ to obtain that
\begin{equation*}
T(\openball{y}{\varepsilon})=S(\openball{y}{\varepsilon}) \text{ for all }y\in\support{S},
\end{equation*} 
which implies that $\cardinality{\support{T}}\geq 2$, that is, $x\not\in A$, which proves the claim.

\begin{claim}
For each $x\in I\setminus C$ there exists an interval $J\subset I$ such that $x\in J$ and $f|_J$ admits a differentiable selection.
\end{claim}
Suppose that $x\in I\setminus A$, then $\cardinality{\support{f(x)}}\geq 2$.
It follows from Lemma \ref{existence of splitting} that there exists an interval $J\subset I$ such that $f$ splits on $J$ as functions $g_i:J\rightarrow\qspace{q_i}{\real^n}$ for $i=1,\ldots,k$ with $k\geq 2$ and $1\leq q_i<q$.
From Lemma \ref{splitting preserves affine approximation} we have that each $g_i$ is affinely approximable on $J$.
Hence, the induction hypothesis gives us a differentiable selection for each $g_i$ which naturally defines a differentiable selection for $f|_J$.

Now, suppose that $x\in B$, then Lemma \ref{existence of pointed splitting} gives us an interval $\tilde{J}\subset I$ such that $f$ admits a pointed split on $(\tilde{J},x)$ defined by functions $\tilde{g}_i:\tilde{J}\rightarrow\qspace{\tilde{q}_i}{\real^n}$ for $i=1,\ldots,\tilde{k}$ with $\tilde{k}\geq 2$ and $1\leq \tilde{q}_i<q$.
From Lemma \ref{pointed splitting preserves affine approximation} we have that each $\tilde{g}_i$ is affinely approximable on $\tilde{J}$.
Hence, the induction hypothesis gives us a differentiable selection for each $\tilde{g}_i$ which naturally defines a differentiable selection for $f|_{\tilde{J}}$ and concludes the proof of the claim.

Each connected component $J\subset I\setminus C$ is an open interval relative to I so we use the previous claim together with Lemma \ref{local to global on closed interval} to find differentiable selections $h_{J,1},\ldots,h_{J,q}:J\rightarrow\real^n$ of $f|_J$.
To extend these selection across $C$, we note that for each $x\in C$ we have that
\begin{equation*}
f(x)=q\di{y(x)}
\end{equation*}
for some $y(x)\in\real^n$ and
\begin{equation*}
Df(x)=q\di{M(x)}
\end{equation*}
for some $M(x)\in\homomorphism{\real}{\real^n}$.

Finally, we define $h_1,\ldots,h_q:I\rightarrow \real^n$ as
\begin{equation*}
h_\alpha(x)=\left\{
\begin{aligned}
 &h_{J,\alpha}(x) && , \text{ if }J\text{ is a connected component of }I\setminus C \text{ and } x\in J\text{ or }\\
 &y(x)            &&, \text{ if }x\in C.
\end{aligned}\right.
\end{equation*}
It is clear that $h=(h_\alpha)_{\alpha=1,\ldots,q}$ defines a selection of $f$ on $I$ and $h$ is differentiable at every $x\in I\setminus C$.
It remains to prove that $h$ is differentiable at points in $C$.
Indeed, if $c\in C$ we note that for any permutation $\sigma\in\perm{q}$ we have $h_{\sigma(\alpha)}(c)=y(c)$ and $Df(c)=\sum_{\alpha=1}^q\di{L_\alpha}=q\di{M(c)}$ so that
\begin{equation*}
\begin{aligned}
\fmetric(f(x),\pfa{c}{f}(x)) & = \sum_{\alpha=1}^q\|h_\alpha(x)-h_{\sigma(\alpha)}(c)-\langle x-c,L_{\sigma(\alpha)}\rangle\|\\
                              & = \sum_{\alpha=1}^q\|h_\alpha(x)-y(c)-\langle x-c,M(c)\rangle\|\\
                              & = \sum_{\alpha=1}^q\|h_\alpha(x)-h_\alpha(c)-\langle x-c,M(c)\rangle\|,
\end{aligned}
\end{equation*}
which proves that each $h_\alpha$ is differentiable at $c$ with $Dh_\alpha(c)=M(c)$.
\end{proof}

\begin{corollary}\label{existence of selection of class 1}
Let $n,q\in\pinteger$, $I\subset\real$ be an interval and $f:I\rightarrow\qspace{q}{\real^n}$ be a multiple-valued function.
If $f$ is affinely approximable and $Df$ is continuous, then $f$ admits a selection of class $1$.
\end{corollary}
\begin{proof}
It follows directly from Theorem \ref{main theorem 1} and Proposition \ref{improvement from differentiable to class 1}.
\end{proof}

\section{Multiple-Valued Functions of Class $1$}


We will use the one-dimensional results of the previous section to prove that all possible definitions of multiple-valued functions of class $1$ are equivalent.
Our approach may be seen of a multiple-valued equivalent of proving that continuity of all partial derivatives implies continuity of the derivative map.

\begin{lemma}[Chain Rule]\label{chain rule}
Let $m,n,q\in\pinteger$, $\delta\in(0,\infty)$, $\Omega\subset\real^m$ be an open set, $f:\Omega\rightarrow\qspace{q}{\real^n}$ be a multiple-valued function and $\gamma:(-\delta,\delta)\rightarrow\Omega$ be a curve of class $1$.
If $f$ is affinely approximable at $\gamma(0)$ with
\begin{equation*}
\tfa{f}(\gamma(0))=\sum_{\alpha=1}^q\di{(y_\alpha,L_\alpha)},
\end{equation*}
then $h=f\circ\gamma$ is affinely approximable at $0$ with
\begin{equation*}
\tfa{h}(0)=\sum_{\alpha=1}^q\di{(y_\alpha,\langle \dot{\gamma}(0),L_\alpha\rangle)}.
\end{equation*}
In particular,
\begin{equation*}
Dh(0)=\up{D}^1_{\gamma(0)}f(\dot{\gamma}(0)).
\end{equation*}
If in addition $f$ is affinely approximable on an open neighbourhood of $\gamma(0)$ and $Df$ is continuous at $0$, then $h$ is affinely approximable on an open neighbourhood of $0$ and $Dh$ is continuous.
\end{lemma}
\begin{proof}
Let $f_1,\ldots,f_q:\Omega\rightarrow \real^n$ be an arbitrary selection of $f$ and $L_1,\ldots,L_q\in\homspace{\real^m}{\real^n}$ be such that
\begin{equation*}
\tfa{f}(\gamma(0))=\sum_{\alpha=1}^q\di{(f_\alpha(\gamma(0)),L_\alpha)}.
\end{equation*}
For each $s\in(-\delta,\delta)$, take $\sigma_s\in\perm{q}$ to be a permutation satisfying
\begin{equation*}
\fmetric(f(\gamma(s),\pfa{\gamma(0)}{f}(\gamma(s)))=\sum_{\alpha=1}^q\|f_{\sigma_s(\alpha)}(\gamma(s))-f_\alpha(\gamma(0))-\langle \gamma(s)-\gamma(0),L_\alpha\rangle \|.
\end{equation*}
Define $\Phi\in\qspace{q}{\polyspace{1}{\real}{\real^n}}$ as
\begin{equation*}
\Phi=\sum_{\alpha=1}^q\di{(f_\alpha(0),\langle \dot{\gamma}(0),L_\alpha\rangle)}
\end{equation*}
and compute
\begin{equation*}
\begin{aligned}
|s|^{-1}\fmetric(h(s),\up{P}_{0,\Phi}(s)) & \leq \sum_{\alpha=1}^q|s|^{-1}\|h_{\sigma_s(\alpha)}(s)-f_\alpha(\gamma(0))-s\langle \dot{\gamma}(0),L_\alpha\rangle \|\\
                                           &\leq\sum_{\alpha=1}^q |s|^{-1} \|f_{\sigma_s(\alpha)}(\gamma(s))-f_\alpha(\gamma(0))-\langle \gamma(s)-\gamma(0),L_\alpha\rangle\|\\
																					 &\quad+ |s|^{-1}\|\langle\gamma(s)-\gamma(0)-s\dot{\gamma}(0),L_\alpha\rangle\|\\
																					\leq&\sum_{\alpha=1}^q \left\|\frac{\gamma(s)-\gamma(0)}{s}\right\|\|\gamma(s)-\gamma(0)\|^{-1}\fmetric(f(\gamma(s)),\pfa{\gamma(0)}{f}(\gamma(s)))\\
																					&\quad + \|L_\alpha\|_{\odot^1}\||s^{-1}|\|\gamma(s)-\gamma(0)-s\dot{\gamma}(0)\|
\end{aligned}
\end{equation*}
Since $\gamma$ is differentiable, it follows that the right hand side tends to $0$ as $s$ tends to $0$.
Thus proving that $\tfa{h}(0)=\Phi$.

\end{proof}

\begin{lemma}\label{lemma derivative}
Let $m,n,q\in\pinteger$, $\Omega\subset\real^m$ be an open set, $f,g:\Omega\rightarrow\qspace{q}{\real^n}$ be multiple-valued functions and $a,b\in\Omega$.
If $f,g$ are affinely approximable at $a,b$ respectively, then
\begin{equation*}
\fmetric_{\real^n}(\up{D}^1_af(v),\up{D}^1_bg(v))\leq\|v\|\fmetric_{\homomorphism{\real^m}{\real^n}}(Df(a),Dg(b)).
\end{equation*}
\end{lemma}
\begin{proof}
Suppose $L_\alpha,M_\alpha\in\homspace{\real^m}{\real^n}$ for $\alpha=1,\ldots,q$ are such that
\begin{equation*}
Df(a)=\sum_{\alpha=1}^q\di{L_\alpha} \text{ and } Dg(b)=\sum_{\alpha=1}^q\di{M_\alpha}.
\end{equation*}
Take $\sigma\in\Pi_q$ satisfying
\begin{equation*}
\fmetric(Df(a),Dg(b))=\sum_{\alpha=1}^q\|L_\alpha-M_{\sigma(\alpha)}\|_{\odot^1}.
\end{equation*}
Hence,
\begin{equation*}
\begin{aligned}
\fmetric_{\real^n}(\up{D}^1_af(v),\up{D}^1_bg(v)) & \leq \sum_{\alpha=1}^q\|\langle v,L_\alpha\rangle-\langle v,M_{\sigma(\alpha)}\rangle\|\\
                                    &\leq \|v\|\sum_{\alpha=1}^q\|L_\alpha - M_{\sigma(\alpha)}\|_{\odot^1}\\
																		& = \|v\|\fmetric(Df(a),Dg(b)).
\end{aligned}
\end{equation*}
\end{proof}

The next Lemma may be compared with \cite{almgren2000}*{Proposition 1.10} when the domain is one-dimensional.
We make rougher estimates and obtain a different constant but we also obtain a finer intermediary result.

\begin{lemma}\label{continuous selection estimate}
Let $m,n,q\in\pinteger$ and let $X\subset\real^m$ be a connected set.
Suppose $g,h:X\rightarrow (\real^n)^q$ are continuous selections for $G,H:X\rightarrow\qspace{q}{\real^n}$ respectively.
Then, for any $\alpha,\beta\in\{1,\ldots,q\}$ and $a,b\in X$ we have
\begin{equation*}
\begin{aligned}
\|h_\beta(b)-g_\alpha(a)\|\leq & 2(q!)\sup\{\fmetric(H(x),G(a)):x\in X\}\\
                             & \;+\|h_\beta(a)-g_\alpha(a)\|.
\end{aligned}
\end{equation*}
\end{lemma}
\begin{proof}
For each $\sigma\in\Pi_q$ we define
\begin{equation*}
X_\sigma=\{x\in X:\fmetric(H(x),G(a))=\sum_{\alpha=1}^q\|h_\alpha(x)-g_{\sigma(\alpha)}(a)\|\}
\end{equation*}
It follows that
\begin{enumerate}[(i)]
\item there exists $\sigma_0\in\Pi_q$ such that $a\in X_{\sigma_0}$ and
\item $X_\sigma$ is closed for all $\sigma\in\perm{q}$.
\end{enumerate}
The second property follows from continuity of $h$ and $g$.
\begin{claim}
Either $X_{\sigma_0}= X$ or for each $\sigma\neq\sigma_0$ such that $X_\sigma\neq\emptyset$ there exists a finite sequence of pairwise distinct permutations $\sigma_1,\ldots,\sigma_k\in\Pi_q$ such that $\sigma_k=\sigma$, $\sigma_i\neq\sigma_0$ and $X_{\sigma_{i-1}}\cap X_{\sigma_i}\neq\emptyset$ for all $i=1,\ldots,k$.
\end{claim}
Indeed, define
\begin{equation*}
\begin{aligned}
&\mathcal{A}=\bigcup\{X_{\sigma'}:X_{\sigma'}\neq\emptyset \text{ and the conclusion of the claim is true for }\sigma'\} \text{ and }\\
&\mathcal{B}=\bigcup\{X_{\sigma'}:\sigma'\neq\sigma_0, X_{\sigma'}\neq\emptyset \text{ and the conclusion of the claim is false for }\sigma'\}.
\end{aligned}
\end{equation*}
It follows from property $(ii)$ and finiteness of $\Pi_q$ that both $\mathcal{A}$ and $\mathcal{B}$ are closed.
Since $X=(\mathcal{A}\cup X_{\sigma_0})\cup\mathcal{B}$ and $\mathcal{A}\cup X_{\sigma_0}\neq\emptyset$, we conclude from connectedness of $X$ that $\mathcal{B}=\emptyset$, which proves the claim.

In particular, $b\in X_\sigma$ for some $\sigma\in\Pi_q$ so we can find a sequence $\sigma_1,\ldots,\sigma_k\in\Pi_q$ as in the claim and $x_i\in X_{\sigma_{i-1}}\cap X_{\sigma_i}$ for each $i=1,\ldots, k$.
We compute
\begin{equation*}
\begin{aligned}
\|h_\beta(b)-g_\alpha(a)\| & \leq \|h_\beta(b)-g_{\sigma_k(\beta)}(a)\|\\
                         & \quad +\sum_{i=1}^k\|g_{\sigma_i(\beta)}(a)-h_\beta(x_i)\|+\|h_\beta(x_i)-g_{\sigma_{i-1}(\beta)}(a)\|\\
                         & \quad +\|g_{\sigma_0(\beta)}(a)-h_\beta(a)\|+\|h_\beta(a)-g_\alpha(a)\|\\
												 & = \fmetric(H(b),G(a))+2\sum_{i=1}^k\fmetric(H(x_i),G(a))\\
												 & \quad +\fmetric(H(a),G(a))+\|h_\beta(a)-g_\alpha(a)\|.
\end{aligned}
\end{equation*}
Note that $\sigma_{i}\neq\sigma_0$ for all $i=1,\ldots,k$ so that $k\leq q!-1$, which concludes the proof.
\end{proof}

Finally, we prove an oscilation lemma that will play the analogous role of the mean value theorem in the multiple-valued case.
This is not quite trivial to generalize since there is no definition of $Df$ as limits of ratios even when the domain is one-dimensional.
We recall the notation $\overrightarrow{[a,b]}=\{a+t(b-a)\in \real^n:t\in[0,1]\}$.

\begin{lemma}[Oscilation Lemma]\label{oscilation lemma}
Let $m,n,q\in\pinteger$, $\Omega\subset\real^m$ be an open set, $a\in\Omega$, $r\in(0,\infty)$ and $f:\Omega\rightarrow\qspace{q}{\real^n}$ be a multiple-valued function.
Suppose that $\closedball{a}{r}\subset\Omega$ and $f$ is affinely approximable in $\closedball{a}{r}$.
If $Df$ is continuous on $\closedball{a}{r}$, then there exists $C_1(q)\in(0,\infty)$ such that
\begin{equation*}
\fmetric(f(b),\pfa{a}{f}(b))\leq C_1\|b-a\|\sup\{\fmetric(Df(x),Df(a)):x\in\overrightarrow{[a,b]}\}
\end{equation*}
for all $b\in\closedball{a}{r}$.
\end{lemma}
\begin{proof}
We may assume that $b\neq a$, otherwise the result is trivial.
Define 
\begin{equation*}
H(t)=f(a+t(b-a)) \text{ for }t\in\left(-\frac{r}{|b-a|},\frac{r}{|b-a|}\right).
\end{equation*}
It follows from Lemma \ref{chain rule} that $H$ is affinely approximable and $DH$ is continuous.
We write $\delta=\frac{r+|b-a|}{2|b-a|}\in(1,\frac{r}{|b-a|})$ and let $h:[-\delta,\delta]\rightarrow(\real^n)^q$ be the selection function of class $1$ for $H$ in $[-\delta,\delta]$ given by Corollary \ref{existence of selection of class 1}.
In particular, $(Dh_\alpha)_{\alpha=1,\ldots,q}$ is a continuous selection for $DH$

Next we define 
\begin{equation*}
g_\alpha(t)=h_\alpha(0)+\langle t,Dh_\alpha(0)\rangle \text{ and } G(t)=\sum_{\alpha=1}^q\di{g_\alpha(t)} \text{ for }t\in[-\delta,\delta].
\end{equation*}
We note that
\begin{enumerate}[(i)]
\item $H(1)=f(b)$;
\item $DH(s)=\up{D}^1_{a+s(b-a)}f(b-a)$;
\item $G(1)=\pfa{a}{f}(b)$ and
\item $Dg_\alpha(t)=Dh_\alpha(0)$ for all $\alpha=1,\ldots,q$.
\end{enumerate}
Use the above properties to compute
\begin{equation*}
\begin{aligned}
\fmetric(f(b),\pfa{a}{f}(b)) & = \fmetric(H(1),G(1))\\
                            & \leq \sum_{\alpha=1}^q\|h_\alpha(1)-g_\alpha(1)\|\\
														& = \sum_{\alpha=1}^q\|(h_\alpha-g_\alpha)(1)-(h_\alpha-g_\alpha)(0)\|\\
														& = \sum_{\alpha=1}^q\left\|\int_0^1 Dh_\alpha(\tau)-Dg_\alpha(\tau)\integrald\tau\right\|\\
														& \leq \sum_{\alpha=1}^q\sup\{\|Dh_\alpha(\tau)-Dg_\alpha(\tau)\|:\tau\in[0,1]\}\\
														& = \sum_{\alpha=1}^q\sup\{\|Dh_\alpha(\tau)-Dh_\alpha(0)\|:\tau\in[0,1]\}\\
														& \leq 2q(q!)\sup\{\fmetric(DH(\tau),DH(0)):\tau\in[0,1]\}\\
														& = 2q(q!)\sup\{\fmetric(\up{D}^1_{a+\tau(b-a)}f(b-a),\up{D}^1_{a}f(b-a)):\tau\in[0,1]\}\\
														& \leq 2q(q!)\|b-a\|\sup\{\fmetric(Df(x),Df(a)):x\in\overrightarrow{[a,b]}\},
\end{aligned}
\end{equation*}
where on the second to last line we used Lemma \ref{continuous selection estimate} and on the final line we used Lemma \ref{lemma derivative}.
\end{proof}

We will now prove the main theorem.

\begin{theorem}\label{main theorem}
Let $m,n,q\in\pinteger$, $\Omega\subset\real^m$ be an open set and $f:\Omega\rightarrow\qspace{q}{\real^n}$ be a multiple-valued function.
The following statements are equivalent:
\begin{enumerate}[(1)]
\item $f$ is affinely approximable in $\Omega$ and $Df:\Omega\rightarrow\qspace{q}{\homspace{\real^m}{\real^n}}$ is continuous;
\item There exists $\Phi:\Omega\rightarrow\qspace{q}{\polyspace{1}{\real^m}{\real^n}}$ such that
\begin{enumerate}[(a)]
\item $\up{P}_{a,\Phi(a)}(a)=f(a)$ for all $a\in\Omega$ and
\item For all $K\subset\Omega$ compact set and $\varepsilon\in(0,\infty)$ there exists $\delta\in(0,\infty)$ such that if $a,b\in K$ satisfy $\|b-a\|<\delta$, then
\begin{equation*}
\fmetric(\up{P}_{a,\Phi(a)}(x),\up{P}_{b,\Phi(b)}(x))\leq\varepsilon\|b-a\|
\end{equation*}
for all $x\in\openball{a}{\|b-a\|}$.
\end{enumerate}
\item $f$ is affinely approximable in $\Omega$ and $\tfa{f}:\Omega\rightarrow\qspace{q}{\polyspace{1}{\real^m}{\real^n}}$ is continuous;
\end{enumerate}
\end{theorem}
\begin{proof}
We begin by proving that $(1)\Rightarrow (2)$ we first define $\Phi=\tfa{f}$ so that 
\begin{equation*}
\up{P}_{a,\Phi(a)}=\pfa{a}{f}
\end{equation*}
and property $(2)(a)$ follows trivially.

To prove property $(2)(b)$, let $K\subset\Omega$ be a compact set and fix $\varepsilon>0$.
Since $Df$ is continuous there exists $\delta>0$ such that for all $z\in K$ and $y\in\closedball{z}{2\delta}$ we have
\begin{equation*}
\fmetric(Df(y),Df(z))<\frac{\varepsilon}{3C_1},
\end{equation*}
where $C_1$ is the constant given by Lemma \ref{oscilation lemma}.
If $b,a\in K$ satisfy $\|b-a\|<\delta$ and $x\in\openball{a}{\|b-a\|}$, then $\|x-b\|<2\|b-a\|<2\delta$ and Lemma \ref{oscilation lemma} implies
\begin{equation*}
\begin{aligned}
\fmetric(\pfa{b}{f}(x),\pfa{a}{f}(x)) & \leq \fmetric(\pfa{b}{f}(x),f(x))+\fmetric(f(x),\pfa{a}{f}(x))\\
                                  & \leq C_1\|x-b\|\sup\{\fmetric(Df(x),Df(b)):x\in\closedball{b}{\|x-b\|}\}\\
																	&\quad +C_1\|x-a\|\sup\{\fmetric(Df(x),Df(a)):x\in\closedball{a}{\|x-a\|}\}\\
																	& \leq C_1\|b-a\|\sup\{\fmetric(Df(x),Df(b)):x\in\closedball{b}{2\delta}\}\\
																	& \quad +C_12\|b-a\|\sup\{\fmetric(Df(x),Df(a)):x\in\closedball{a}{\delta}\}\\
																	& \leq \|b-a\|\varepsilon.
\end{aligned}
\end{equation*}

Let us prove that $(2)\Rightarrow(3)$.
Let $a\in\Omega$ be arbitrary and $x\in\Omega$ be sufficiently close to $a$.
It follows from (2)(a) that
\begin{equation*}
\|x-a\|^{-1}\fmetric(f(x),\up{P}_{a,\Phi(a)}(x))=\|x-a\|^{-1}\fmetric(\up{P}_{x,\Phi(x)}(x),\up{P}_{a,\Phi(a)}(x)).
\end{equation*}
Property (2)(b) implies that $f$ is affinely approximable at $a$ and $\tfa{f}(a)=\Phi(a)$.

It remains to prove that $\tfa{f}$ is continuous.
Fix $b\in \Omega$ and $\varepsilon>0$.
Let $\delta_0>0$ be sufficiently small so that $\closedball{b}{2\delta_0}\subset\Omega$, $K=\closedball{b}{2\delta_0}$ and let $\Gamma_3>0$ be given by \ref{multivalued combinatorial lemma} and $\delta_1>0$ be given by property $(2)(b)$ with respect to $K$ and $\frac{\varepsilon}{2\Gamma_3}$.
First we write
\begin{equation*}
\Phi(b)=\sum_{\alpha=1}^q\di{\psi_\alpha}
\end{equation*}
with $\psi_\alpha=(y_\alpha,L_\alpha)$ and take $\delta_2>0$ sufficiently small depending on $b$ so that
\begin{equation*}
\sum_{\alpha=1}^q\|\langle a-b,L_\alpha\rangle\|\leq\frac{\varepsilon}{2}, \text{ whenever }\|b-a\|<\delta_2.
\end{equation*}
Now, take $a\in K$ satisfying $\|b-a\|<\min\{\delta_0,\delta_1,\delta_2,1\}$ and write
\begin{equation*}
\Phi(a)=\sum_{\alpha=1}^q\di{\phi_\alpha}.
\end{equation*}
Recall that for any $x\in\real^m$,
\begin{equation*}
D^0\psi_\alpha(x)=\psi_{\alpha}(0)+\langle x,L_\alpha\rangle \text{ and } D^1\psi_\alpha(x)=D^1\phi_\alpha(0).
\end{equation*}
We compute:
\begin{equation*}
\begin{aligned}
\fmetric(\Phi(b),\Phi(a))& =\inf\{\sum_{\alpha=1}^q\|\phi_\alpha-\psi_{\sigma(\alpha)}\|_{\mathbb{P}^1}:\sigma\in\perm{q}\}\\
                   & = \inf\{\sum_{\alpha=1}^q\sum_{j=0}^1\|D^j\phi_\alpha(0)-D^j\psi_{\sigma(\alpha)}(0)\|_{\odot^j}:\sigma\in\perm{q}\}\\
									 & = \inf\{\sum_{\alpha=1}^q\|D^0\phi_\alpha(0)-\psi_{\sigma(\alpha)}(0)\|_{\odot^0}\\
									 & \quad +\|D^1\phi_\alpha(0)-D^1\psi_{\sigma(\alpha)}(0)\|_{\odot^1}:\sigma\in\perm{q}\}\\
									& = \inf\{\sum_{\alpha=1}^q\|D^0\phi_{\alpha}{(0)}-\psi_{\sigma(\alpha)}(0)\pm\langle a-b,L_{\sigma(\alpha)}\rangle\|_{\odot^0}\\
									&\quad +\|D^1\phi_\alpha(0)-D^1\psi_{\sigma(\alpha)}(a-b)\|_{\odot^1}:\sigma\in\perm{q}\}\\
									& \leq \inf\{\sum_{\alpha=1}^q\|\langle a-b,L_\alpha\rangle\|_{\odot^0}\\
									&\quad +\sum_{j=0}^1\|D^j\phi_\alpha(0)-D^j\psi_{\sigma(\alpha)}(a-b)\|_{\odot^j}:\sigma\in\perm{q}\}\\
									 \leq \frac{\varepsilon}{2}+\inf\{&\sum_{\alpha=1}^q\sum_{j=0}^1\|D^j\phi_\alpha(0)-D^j\psi_{\sigma(\alpha)}(a-b)\|_{\odot^j}:\sigma\in\perm{q}\}\\
									= \frac{\varepsilon}{2}+\|b-a\|^{-1}\inf\{&\sum_{\alpha=1}^q\sum_{j=0}^1\|b-a\|\|D^j\phi_\alpha(0)-D^j\psi_{\sigma(\alpha)}(a-b)\|_{\odot^j}:\sigma\in\perm{q}\}\\
									 \leq \frac{\varepsilon}{2}+\|b-a\|^{-1}\inf\{&\sum_{\alpha=1}^q\sum_{j=0}^1\|b-a\|^j\|D^j\phi_\alpha(0)-D^j\psi_{\sigma(\alpha)}(a-b)\|_{\odot^j}:\sigma\in\perm{q}\}\\
									 \leq \frac{\varepsilon}{2}+ \Gamma_3\|b-a\|^{-1}&\sup\{\fmetric(\up{P}_{a,\Phi(a)}(x),\up{P}_{b,\Phi(b)}(x)):x\in\openball{a}{\|b-a\|}\}\\
									& \leq \varepsilon
\end{aligned}
\end{equation*}
where in the second to last line we used Lemma \ref{multivalued combinatorial lemma} with $X=\openball{a}{\|b-a\|}$ and in the final line we used property $(2)(b)$.

The case $(3)\Rightarrow (1)$ is trivial and concludes the proof.
\end{proof}

In view of the previous Theorem any of the above properties may be used as a definition of multiple-valued functions of class $1$.

\begin{definition}
Let $m,n,q\in\pinteger$, $\Omega\subset\real^m$ and $f:\Omega\rightarrow\qspace{q}{\real^n}$ be a multiple-valued function.
We say that $f$ is of class $1$ if and only if $f$ is affinely approximable at every point of $\Omega$ and $Df:\Omega\rightarrow\qspace{q}{\homspace{\real^m}{\real^n}}$ is continuous.
\end{definition}

With very little modification in the above we can prove the H\"older differentiable version.

\begin{theorem}
Let $m,n,q\in\pinteger$, $\lambda\in(0,1]$, $\Omega\subset\real^m$ be an open set and $f:\Omega\rightarrow\qspace{q}{\real^n}$ be a multiple-valued function.
The following statements are equivalent:
\begin{enumerate}[(1)]
\item $f$ is of class $1$ and $Df:\Omega\rightarrow\qspace{q}{\homspace{\real^m}{\real^n}}$ is locally $\lambda$-H\"older continuous;
\item There exists $\Phi:\Omega\rightarrow\qspace{q}{\polyspace{1}{\real^m}{\real^n}}$ such that
\begin{enumerate}[(a)]
\item $\up{P}_{a,\Phi(a)}(a)=f(a)$ for all $a\in\Omega$ and
\item For all $K\subset\Omega$ compact set there exists $C\in(0,\infty)$ such that if $a,b\in K$, then
\begin{equation*}
\fmetric(\up{P}_{a,\Phi(a)}(x),\up{P}_{b,\Phi(b)}(x))\|b-a\|^{-1-\lambda}<C
\end{equation*}
for all $x\in\openball{a}{\|b-a\|}$.
\end{enumerate}
\item $f$ is affinely approximable in $\Omega$, $\tfa{f}:\Omega\rightarrow\qspace{q}{\polyspace{1}{\real^m}{\real^n}}$ is continuous and for every $a\in\Omega$ there exist $\delta,C\in(0,\infty)$ such that if $b\in\openball{a}{\delta}$ then
\begin{equation*}
\inf\{\sum_{\alpha=1}^q\frac{\|y_\alpha(b)-y_{\sigma(\alpha)}(a)\|}{\|b-a\|}+\frac{\|L_\alpha(b)-L_{\sigma(\alpha)}(a)\|}{\|b-a\|^\lambda}:\sigma\in\perm{q}\}\leq C,
\end{equation*}
where $(y_\alpha,L_\alpha)_{\alpha=1,\ldots,q}:\Omega\rightarrow {\polyspace{1}{\real^m}{\real^n}}^q$ is an arbitrary selection of $\tfa{f}$.
\end{enumerate}
\end{theorem}

As before we are now able to define H\"older differentiable multiple-valued functions.
We unorthodoxly conclude this article with this definition.

\begin{definition}
Let $m,n,q\in\pinteger$, $\lambda\in(0,1]$, $\Omega\subset\real^m$ and $f:\Omega\rightarrow\qspace{q}{\real^n}$ be a multiple-valued function.
We say that $f$ is of class $(1,\lambda)$ if and only if $f$ is of class $1$ and $Df:\Omega\rightarrow\qspace{q}{\homspace{\real^m}{\real^n}}$ is $\lambda$-H\"older continuous in $\Omega$.
\end{definition}

\bibliographystyle{plain}

\end{document}